\documentclass[a4paper, oneside,11pt]{article}

\usepackage[latin1]{inputenc}
\usepackage[T1]{fontenc}
\usepackage[english]{babel}
\usepackage{amssymb}
\usepackage{amsmath}
\usepackage{amsthm}
\usepackage{graphicx}

\usepackage[all]{xy}
\usepackage{verbatim}
\usepackage{amsmath}
\usepackage{amsthm}
\usepackage{graphicx}
\usepackage{amssymb}
\usepackage{epstopdf}
\usepackage{url}
\usepackage{amsfonts}
\usepackage{todonotes}
\newcommand{\mc}{\mathcal}

\newcommand{\pt}{\partial}

\newcommand{\br}{\mathbb{R}}

\renewcommand{\(}{\left(}
\renewcommand{\)}{\right)}
\renewcommand{\[}{\left[}
\renewcommand{\]}{\right]}
\newcommand{\na}{\nabla}

\newtheorem{thm}{Theorem}
\newtheorem{lem}[thm]{Lemma}

\newtheorem{prop}[thm]{Proposition}

\newtheorem{remark}[thm]{Remark}

\def\be{\begin{equation}}
\def\ee{\end{equation}}
\def\bea{\begin{eqnarray}}
\def\eea{\end{eqnarray}}

\usepackage{enumerate}

\newcommand{\ds}{\displaystyle}

\usepackage{bbm}

\numberwithin{thm}{section}
\numberwithin{equation}{section}

\newcommand{\RR}{\mathbb{R}}

\renewcommand{\div}{\operatorname{div}}

\usepackage{hyperref}
\usepackage{mathtools}
\mathtoolsset{showonlyrefs}

\newcommand*\di{\mathop{}\!\mathrm{d}}

\title{Global strong solutions in $\RR^3$ for ionic Vlasov-Poisson systems}

\author{
Megan Griffin-Pickering 
  \thanks{Durham University, Department of Mathematical Sciences, Lower Mountjoy, Stockton Road, Durham DH1 3LE, UK. \newline Email: \textsf{megan.k.griffin-pickering@durham.ac.uk}}
  \and
Mikaela Iacobelli
  \thanks{ETH Zurich, Department of Mathematics, Ramistrasse 101, 8092 Zurich, Switzerland.\newline Email: \textsf{mikaela.iacobelli@math.ethz.ch}}
}

\begin{document}
\maketitle

\begin{abstract}
Systems of Vlasov-Poisson type are kinetic models describing dilute plasma.
The structure of the model differs according to whether it describes the electrons or positively charged ions in the plasma.
In contrast to the electron case, where the well-posedness theory for Vlasov-Poisson systems is well established, the well-posedness theory for ion models has been investigated more recently.
In this article, we prove global well-posedness for two Vlasov-Poisson systems for ions, posed on the whole three-dimensional Euclidean space $\mathbb{R}^3$, under minimal assumptions on the initial data and the confining potential.

\end{abstract}

\section{Introduction}

In this article, we investigate the well-posedness theory of a kinetic model for the ions in a dilute plasma. Plasma is an ionised gas, which forms when an electrically neutral gas is subjected to a high temperature or a strong electromagnetic field. This causes the gas particles to dissociate: electrons split apart from the rest of the gas particle.
A plasma therefore contains two distinguished types of charged particle: negatively charged electrons and positively charged ions. 

The Vlasov-Poisson system is a well established kinetic model used to describe plasma. 
The version of the system that has been most widely discussed in the mathematics literature is a model for the electrons in the plasma, evolving against a background of ions that is presumed to have a given stationary distribution. This model takes the following form:
\be \label{eq:VP}
(VP) : = 
\begin{cases} \ds 
\partial_t f_e + v \cdot \nabla_x f_e + \frac{q_e}{m_e} E \cdot \nabla_v f_e = 0, \\ \ds
\nabla_x \times E = 0, \quad
\epsilon_0 \nabla_x \cdot E = q_i \rho[f_i] + q_e \rho[f_e] , \\ \ds
\rho[f_e](t,x) : = \int_{\RR^d} f_e(t,x,v) \di v  ,\\ \ds
f_e(0,x,v) = f_{e,0}(x,v) \geq 0 .
\end{cases}
\ee
Here $f_e(t,x,v)$ represents the phase-space density of electrons, $q_e$ and $q_i$ denote respectively the charge on each electron and each ion, $m_e$ is the mass of an electron, $\epsilon_0$ is the vacuum permittivity, and $\rho[f_i](x)$ denotes the spatial density of ions which is assumed to be given and independent of time. The assumption that the ion distribution is stationary is justified by the fact that the mass of an ion is typically much greater than the mass of an electron. It is therefore common to make the approximation that the ions are stationary and uniformly distributed.

In this article, we instead consider a Vlasov-Poisson type system, extensively used in physics, describing the ions in a plasma. In analogy with the electron model \eqref{eq:VP}, we consider a system of the form
\be \label{eq:VP-ions-physical}
\begin{cases} \ds 
\partial_t f_i + v \cdot \nabla_x f_i + \frac{q_i}{m_i} E \cdot \nabla_v f_i = 0, \\ \ds
\nabla_x \times E = 0, \quad
\epsilon_0 \nabla_x \cdot E = q_i \rho[f_i] + q_e \rho[f_e] , \\ \ds
\rho[f_i](t,x) : = \int_{\RR^d} f_i(t,x,v) \di v  ,\\ \ds
f_i(0,x,v) = f_{i,0}(x,v) \geq 0 .
\end{cases}
\ee
To complete this model, it is necessary to specify the electron distribution $\rho[f_e]$. A widely used assumption is that the electrons are in thermal equilibrium. This is justified by the fact that the electrons are relatively very light and so fast moving, with a significant collision frequency. Thus the equilibrium distribution is a Maxwell-Boltzmann law of the form
\be
\rho[ f_e ] \sim e^{- \beta_e q_e \Phi},
\ee
where the ambient electrostatic potential $\Phi$ is defined to be a function such that $E = - \nabla_x \Phi$, while $\beta_e$ denotes the inverse electron temperature.

After an appropriate rescaling, this choice of electron distribution results in the following system:
\be \label{eq:VP-ions}
\begin{cases} \ds 
\partial_t f + v \cdot \nabla_x f + E \cdot \nabla_v f = 0, \\ \ds
E = - \nabla_x U, \quad
\Delta U = A(t, U) \, e^U - \rho[f], \\ \ds
\rho[f](t,x) : = \int_{\RR^d} f(t,x,v) \di v  ,\\ \ds
f(0,x,v) = f_{0}(x,v) \geq 0,  \int_{\br^{2d}} f_{0} \di x \di v=1 .
\end{cases}
\ee
Here $A(t, U) > 0$ is a scaling term in the electron distribution, which we will discuss further below.

It is natural to include a further spatial confinement of the electrons, using an external potential. That is, we assume that the electrons are also subject to a given external potential $H$. Their thermal equilibrium is then of the form
\be
\rho[ f_e ] \sim e^{- H + U} = g e^U,
\ee
where the function $g: \RR^d \to [0, + \infty)$ is defined by $g : = e^{-H}$. We assume throughout the paper a minimal condition on $g,$ namely that $g$ is fixed and belongs to the space $L^1 \cap L^\infty(\RR^d)$.

We consider the two most natural versions of the Vlasov-Poisson system for ions. These differ based on the choice of the scaling $A$.
Choosing $A=1$ results in the following system:
\begin{equation}
\label{eq:vpme-var}
(VPME)_V := \left\{ \begin{array}{ccc}\pt_t f+v\cdot \nabla_x f+ E\cdot \nabla_v f=0,  \\
E=-\nabla U, \\
\Delta U= g e^{U}- \rho_f,\\
f\vert_{t=0}=f_{0}\ge0,\ \  \int_{\br^{2d}} f_{0}\,dx\,dv=1.
\end{array} \right.
\end{equation}
Note that for solutions of \eqref{eq:vpme-var}, the total charge is not necessarily conserved and the system therefore may not be globally neutral at all times. An alternative choice is to enforce global neutrality. For this $A$ must be chosen to normalise the electron distribution, that is,  $A = \left ( \int_{\br^d} g e^{U} \di x\right )^{-1}$, which results in the following alternative system:
\begin{equation}
\label{eq:vpme-fixed}
(VPME)_F := \left\{ \begin{array}{ccc}\pt_t f+v\cdot \nabla_x f+ E\cdot \nabla_v f=0,  \\
E=-\nabla U, \\
\Delta U= \frac{g e^{U}}{\int_{\br^d} g e^{U} \di x} - \rho_f,\\
f\vert_{t=0}=f_{0}\ge0,\ \  \int_{\br^{2d}} f_{0}\,dx\,dv=1.
\end{array} \right.
\end{equation}

Both systems are usually referred to as the Vlasov-Poisson system \textit{with massless electrons}, abbreviated to VPME. This refers to the fact that these systems can be derived from a coupled system of ions and electrons in the limit of `massless electrons', in which the ratio of the electron and ion masses, $\frac{m_e}{m_i}$, tends to zero.
For example, Bardos, Golse, Nguyen and Sentis \cite{BGNS18} discuss this limit for coupled kinetic systems of the form
\be \label{eq:VP-coupled}
\begin{cases}
\partial_t f_i + v \cdot \nabla_x f_i + \frac{q_i}{m_i} E \cdot \nabla_v f_i = 0, \\
\partial_t f_e + v \cdot \nabla_x f_e + \frac{q_e}{m_e} E \cdot \nabla_v f_e = C(m_e) Q(f_e), \\
\nabla_x \times E = 0, \quad
\epsilon_0 \nabla_x \cdot E = q_i \rho[f_i] + q_e \rho[f_e] .
\end{cases}
\ee
In the equation above, $Q$ represents a collision operator such as a Boltzmann or BGK operator.
Under the assumption that sufficiently regular solutions of this system exist, they identify that in the limit the electrons indeed assume a Maxwell-Boltzmann law, leading to a model for the ions that is similar to \eqref{eq:vpme-var}, but with a time-dependent electron temperature. Also, for such a model, they prove a global well-posedness result on the torus.

Systems of the form \eqref{eq:VP-ions} have been used in astrophysics literature, for example in studies of the expansion of plasma into vacuum \cite{Medvedev2011}, numerical investigations of the formation of ion-acoustic shocks \cite{Mason71, SCM} and of the phase-space vortices that form behind these shocks \cite{BPLT1991}. 

In this article, we consider the well-posedness of both \eqref{eq:vpme-var} and \eqref{eq:vpme-fixed}.
We remark that the well-posedness theory for Vlasov-Poisson-type systems heavily depends on the dimension $d$ in which the problem is posed and on the boundary conditions imposed on the system. Two frequently considered boundary conditions are the periodic case, in which the system is posed on the $d$-dimensional flat torus, and the whole space case, in which the problem is posed on all of $\RR^d$ with a condition that $f$ and $E$ decay at infinity.

\begin{remark}
Note that for the Vlasov-Poisson system for ions on the torus, the external confining potential $H$ is not typically used (in other words, $g \equiv 1$). Moreover one may take $A=1$ without loss of generality, since changing $A$ corresponds to adding a constant to $U$. On the torus, the Poisson equation
\be
\Delta U = h
\ee
has a solution only if $h$ has total integral zero; it follows that if a solution of \eqref{eq:vpme-var} on the torus exists, it must necessarily be globally neutral at all times. Thus on the torus there is no distinction between the system \eqref{eq:vpme-var} with variable total charge and the system \eqref{eq:vpme-fixed} with fixed total charge.
\end{remark}

In one dimension ($d=1$), global well-posedness for VPME was proved by Han-Kwan and the second author \cite{IHK1}.
In dimension $d=3$, Bouchut \cite{Bouchut} proved that global weak solutions exist for both systems \eqref{eq:vpme-var} and \eqref{eq:vpme-fixed} on the whole space.
In a recent work \cite{IGP-WP}, the authors proved global well-posedness for the Vlasov-Poisson system for ions in dimension $d=2$ and $d=3$ in the periodic case, i.e. when the problem is posed on the flat torus, with $g \equiv 1$. However, a similar well-posedness result was not previously available for the whole space case. This is the goal of this work.

To make a parallel, the classical Vlasov-Poisson system for electrons is known to be globally well-posed in dimension $d \leq 3$.
Global existence of classical solutions in the whole space or on the torus was shown in, for example, \cite{Ukai-Okabe, Pfaffelmoser, Lions-Perthame, Batt-Rein}. Uniqueness for solutions with bounded density was proved by Loeper \cite{Loeper}.
For a more detailed account of the development of this theory, see for example \cite{WP-proceedings}.

\subsection{Main Result}

The main result of this paper is a global well-posedness result for the VPME systems \eqref{eq:vpme-var} and \eqref{eq:vpme-fixed} in $\RR^3$, under minimal assumptions on the initial data and the confining potential $g$.
To state the main theorem, we first define the energy functionals associated to each of the systems \eqref{eq:vpme-var} and \eqref{eq:vpme-fixed}. Each of these functionals is conserved by sufficiently regular solutions of the associated system.
For system \eqref{eq:vpme-var} where the total charge is variable, we use the following functional:
\be
\mc{E}_V [f] : = \int_{\br^3 \times \br^3} |v|^2 f \di x \di v + \int_{\br^3} |E|^2 \di x + 2 \int_{\br^3} (U - 1) g e^{U} \di x.
\ee
For system \eqref{eq:vpme-fixed}, with fixed total charge, we use
\be
\mc{E}_F [f] : = \int_{\br^3 \times \br^3} |v|^2 f \di x \di v + \int_{\br^3} |E|^2 \di x + 2 \int_{\br^3} \phi g e^{\phi} \di x,
\ee
where
\be
\phi = U - \log{\left ( \int_{\br^3} g e^U \di x \right )} .
\ee
The following theorem is the main result of this article.

\begin{thm}[Global well-posedness]
\label{thm:wp}
Let ${f_0 \in L^1\cap L^\infty(\br^3 \times \br^3)}$ be a probability density 
satisfying
$$
\int_{\br^3\times \br^3}|v|^{m_0}f_0(x,v)\,dx\,dv <+\infty \,\,\,\text{ for some $m_0>6$},\qquad f_0(x,v)\leq \frac{C}{(1+|v|)^r} \,\,\,\text{ for some $r>3$}.
$$
Assume that $g\in L^1\cap L^\infty(\br^3)$, with $g\geq 0$ satisfying $\int_{\br^3}g=1$, and that 
$\mc{E}_V[ f_0] \leq C$ (resp.  $\mc{E}_F[ f_0] \leq C$). Then there exists a unique solution ${f\in L^\infty([0,T] ; L^1\cap L^\infty(\br^3 \times \br^3))}$ of  \eqref{eq:vpme-var}  (resp. \eqref{eq:vpme-fixed}) with initial datum $f_0$ such that $\rho_f \in L^{\infty}([0,T] ; L^\infty(\br^3))$.
\end{thm}

\begin{remark}
We refer to the class of solutions constructed here -- bounded distributional solutions $f$ of 
\eqref{eq:vpme-var} and \eqref{eq:vpme-fixed} whose density $\rho_f$ is uniformly bounded -- as {\it strong} solutions.

Uniqueness holds for this class of solutions, as we prove below in Theorem~\ref{thm:uniq}.
Moreover, for strong solutions the electric field has at least a log-Lipschitz regularity, and so the associated characteristic flow is well-defined.
In particular it is possible to show that, if the initial datum is additionally $C^1$, then the corresponding solution is also $C^1$. 
Thus our global existence of strong solutions (Theorem~\ref{thm:wp}) also includes, as a byproduct, the global existence of classical solutions.
\end{remark}

\begin{remark}
Instead of assuming $f_0(x,v)\leq \frac{C}{(1+|v|)^r}$ for some $r>3$, one can replace this hypothesis with assumption (10) in \cite[Corollary 3]{Lions-Perthame}.
\end{remark}

\begin{remark}
Our result is essentially optimal in terms of the assumptions.
Indeed, as shown in  \cite[Equation 16]{Lions-Perthame}, controlling moments of order larger than $6$ is needed to guarantee that our solution is strong (i.e., $\rho_f \in L^{\infty}([0,T] ; L^\infty(\br^3))$).
Also, the boundedness of $g$ is needed to ensure that the electric field enjoys at least a log-Lipschitz regularity, so that characteristics exist and are unique.
\end{remark}

\subsection{Strategy}

The first step in the proof is to obtain regularity and stability estimates on the electrostatic potential $U.$ This is carried out in Section~\ref{sec:electric}. This provides the necessary tools to complete the proof of well-posedness.
This is done in two stages. In Section~\ref{sec:uniqueness} we prove uniqueness for solutions with bounded density.
Then, in Section~\ref{sec:moments}, we show the global existence of strong solutions.

\subsubsection{Analysis of the Electrostatic Potential}

The analysis of the VPME systems \eqref{eq:vpme-var} and \eqref{eq:vpme-fixed} hinges on an understanding of the electrostatic potential $U$.

We require two kinds of estimates. On the one hand, we prove regularity estimates that give quantitative bounds on $U$ given certain $L^p$ bounds on $\rho_f$. On the other hand, we prove stability estimates, controlling the distance between the electric fields for two solutions of the VPME system, as is needed to prove the uniqueness of solutions. 

Our strategy is based on the following decomposition of the electric field. We write the electrostatic potential in the form
\be
U = \bar U + \widehat U,
\ee
where $\bar U$ satisfies the equation
\be \label{eq:barU}
- \Delta \bar U = \rho_f, \qquad \lim_{|x| \to 0} \bar U(x) = 0 .
\ee
In other words, $\bar U$ satisfies the same equation as the electrostatic potential in the Vlasov-Poisson system for electrons. The remainder $\widehat U$ must then satisfy either
\be \label{eq:hatU-both}
\Delta \widehat U = g e^{\bar U + \widehat U} \quad \text{or} \quad \Delta \widehat U = \frac{ g e^{\bar U + \widehat U}}{\int_{\br^3} g e^{\bar U + \widehat U} \di x}.
\ee
{
We will similarly write
\be
E = \bar E + \widehat E, \qquad \bar E : = - \nabla \bar U , \quad \widehat E : = - \nabla \widehat U .
\ee }
This decomposition was introduced in \cite{IHK1} in order to study the Vlasov-Poisson system for ions in the one dimensional case. It was then used in \cite{IGP-WP} to study well-posedness in the cases $d=2,3$ on the torus.

In Section~\ref{sec:electric} of this article, we analyse the remainder term $\widehat U$. To deal with the nonlinearity in the equation satisfied by $\widehat U$, we make use of techniques from the calculus of variations.

In this way we are able to show that, under assumptions on $\rho_f$ that we expect to be satisfied by solutions of the Vlasov-Poisson systems \eqref{eq:vpme-var} and \eqref{eq:vpme-fixed}, the equations \eqref{eq:hatU-both} for $\widehat U$ are well-posed and enjoy good regularity estimates.
Specifically, we show that $\widehat U \in C^{1, \alpha}$ for any $\alpha \in (0,1)$,  with a uniform in time bound on this norm.
Thus $U$ is close to $\bar U$ up to a smoother correction, which is controlled uniformly in time in a strong norm.  These estimates then allow methods developed for the Vlasov-Poisson system for electrons to be adapted to the ion case.

This strategy was previously used in \cite{IGP-WP} to show well-posedness on the torus in dimension two and three. Here we apply it to the case where $x \in \RR^3$.
There are two main differences in the whole space case compared to the torus case.
One is that the domain is unbounded and we therefore need to account for the decay of the potential at infinity. 
In practice, this involves identifying an appropriate functional setting for the optimisation problem characterising $\widehat U$.

The other is that in the whole space we study two different models, with different nonlinearities. In particular, for the model \eqref{eq:vpme-fixed} with fixed total charge, the nonlinearity is different from the torus case due to the normalisation of the electron density.
Notice, for instance, that this nonlinearity is not monotone in $\widehat U$.
To handle this nonlinearity with techniques from the calculus of variations, it is necessary to modify the functional used to show the well-posedness of the equation. The natural choice of functional is not bounded below, and so the proof in Subsection~\ref{sec:fixed-exist} differs from the one in \cite{IGP-WP}.
Moreover, a new stability estimate is needed (Lemma~\ref{lem:fixed-stability}).

\subsubsection{Well-posedness in $\RR^3$}
The well-posedness consists of two parts: existence and uniqueness of strong solutions. 
The existence is based on showing the propagation of moments. The idea is to show an a priori estimate on solutions, to the effect that, if the initial datum has a velocity moment of sufficiently high order: if
\be
\int_{\RR^3 \times \RR^3} |v|^{m_0} f_0 \di x \di v < + \infty,
\ee
then the velocity moments of the solution can also be controlled:
\be
\sup_{t \in [0,T]} \int_{\RR^3 \times \RR^3} |v|^{m_0} f(t,x,v) \di x \di v < + \infty .
\ee

In Section~\ref{sec:moments}, we prove the propagation of moments in this sense for the VPME systems \eqref{eq:vpme-var} and \eqref{eq:vpme-fixed}. The principle is to follow the approach of Lions and Perthame \cite{Lions-Perthame}, adapting it to include the extra part of the electrostatic potential $\widehat U$. This is possible thanks to the uniform estimates obtained in Section~\ref{sec:electric}.

For the uniqueness part of Theorem~\ref{thm:wp}, we use an approach in the style of Loeper \cite{Loeper}, who proved uniqueness for solutions of the Vlasov-Poisson system for electrons such that $\rho_f$ is bounded in $L^\infty(\RR^d)$. Loeper's strategy is to prove a stability property for solutions with respect to the initial data, quantified in the second order Wasserstein distance $W_2$. In Section~\ref{sec:uniqueness}, we prove an estimate of this type for the VPME systems \eqref{eq:vpme-var} and \eqref{eq:vpme-fixed} in the whole space.  The key here is that we need to prove suitable stability estimates for the smooth part of the potential $\widehat U$, in the case of the unbounded domain $\RR^3$. We carry this out in Subsection~\ref{sec:electric-stability}.

In Subsection~\ref{sec:thm-proof}, we show how to use these results to complete the proof of Theorem~\ref{thm:wp} -- in particular, to show that under the assumptions of the theorem, the resulting solutions have bounded density so that the uniqueness result may be applied.

\subsection{Proof of the Main Result} \label{sec:thm-proof}

\begin{proof}[Proof of Theorem \ref{thm:wp}]
Arguing as in \cite{Lions-Perthame} and in \cite{IGP-WP}, by approximation one can construct a global solution $f\in L^\infty([0,T] ; L^1\cap L^\infty(\br^3 \times \br^3))$ of \eqref{eq:vpme-var} (resp. \eqref{eq:vpme-fixed}) with uniformly bounded energy. Then, it follows by Proposition \ref{prop:moment-control} below that all moments of order less than $m_0$ are uniformly bounded on every finite time interval.

{
Next, we show that this implies that the solution has bounded density. We introduce the characteristic system associated to the Vlasov-Poisson system, which is the following ODE system: for $(x,v) \in \br^3 \times \br^3$,
\be
\frac{\rm{d}}{\rm{d}t} X(t ,x,v) = V(t , x,v), \; \frac{\rm{d}}{\rm{d}t} V(t , x,v) = E\left ( X(t , x,v) \right ) ; \quad X(0,x,v) = 0 , \,V(0,x,v) = v .
\ee }
As in \cite{Lions-Perthame}, since $m_0>6$ this implies that $\bar E$ is uniformly bounded (see \cite[Equation 16]{Lions-Perthame}), while $\widehat E$ is uniformly bounded thanks to Propositions \ref{prop:variable-hatU-reg}-\ref{prop:fixed-hatU-reg}.
This implies that $E$ is uniformly bounded, and therefore the characteristics satisfy the bound
$$
|V(t,x,v)-v|\leq C_T \qquad \text{for all }(t,x,v) \in [0,T]\times \br^3\times \br^3.
$$
Thus
$$
f(t,X(t,x,v),V(t,x,v))=f_0(x,v)\leq \frac{C}{(1+|v|)^r} \leq \frac{C_T}{(1+|V(t,x,v)|)^r} \qquad \text{for all }(t,x,v) \in [0,T]\times \br^3\times \br^3,
$$
and so
$$
f(t,y,w)\leq \frac{C_T}{(1+|w|)^r} \qquad \text{for all }(t,y,w) \in [0,T]\times \br^3\times \br^3.
$$
Since $r>3$, this yields
$$
\rho_f(t,y) \leq C_T \int_{\br^3}\frac{1}{(1+|w|)^r}\,dw \leq C_T \qquad \text{for all }(t,y) \in [0,T]\times \br^3,
$$
and the uniqueness follows by Theorem \ref{thm:uniq}.

\end{proof}

\subsection{Energy Functionals} \label{sec:energy}

We noted above that each of the VPME systems has an associated energy functional, which we denoted respectively by $\mc{E}_V$ and $\mc{E}_F$. These energy functionals are formally conserved by their associated systems.
The control of these energy functionals implies an integrability bound on the mass density $\rho_f$.

\begin{lem}[Control of the energy implies a moment bound] \label{energy-moment}
Assume one of the conditions
\be
\mc{E}_V[f] \leq C_0, \qquad \mc{E}_F[f] \leq C_0 .
\ee
Then there exists a constant $C$ depending on $C_0$ and $\| g \|_{L^1(\br^3)}$ only such that
\be
\int_{\br^3 \times \br^3} |v|^2 f \di x \di v \leq C. 
\ee
It follows that, if $f_0\in L^{\infty}(\br^3 \times \br^3),$ the associated mass density, $\rho_f = \int_{\br^3} f \di v$ satisfies
\be \label{rho-l53}
\| \rho_f \|_{L^{\frac{5}{3}}(\br^3)} \leq C .
\ee
\end{lem}
\begin{proof}
Observe that the functions $x e^x$, $(x-1)e^x$ are bounded from below, uniformly for all $x \in \br$:
\be
x e^x \geq - e^{-1}, \qquad (x-1) e^x \geq -1 .
\ee
Therefore, since $g \geq 0$, in the variable charge case we have
\be
\int_{\br^3 \times \br^3} |v|^2 f \di x \di v \leq \mc{E}_V[f] + 2 \| g \|_{L^1(\br^3)}  \leq C\left ( C_0, \| g \|_{L^1(\br^3)} \right ).
\ee
In the fixed charge case, we have
\be
\int_{\br^3 \times \br^3} |v|^2 f \di x \di v \leq \mc{E}_V[f] + \frac{2}{e} \| g \|_{L^1(\br^3)}  \leq C\left ( C_0, \| g \|_{L^1(\br^3)} \right ).
\ee
The estimate \eqref{rho-l53} then follows from a standard interpolation argument; see Lemma~\ref{lem:rho-interpolation} below.

\end{proof}

\begin{paragraph}{Notation.}
The notation $L^p(g)$ denotes $L^p$ norms taken with respect to the density $g$:
\be
\| f\|_{L^p(g)}^p = \int_{\br^3} |f(x)|^p g(x) \di x .
\ee
\end{paragraph}

\section{Electric Field Estimates} \label{sec:electric}

\subsection{Decomposition}

We decompose the electrostatic potential $U$ into the form $U = \bar U + \widehat U$, where $ \bar U $ satisfies
\be \label{eq:Poisson-sing}
- \Delta \bar U= \rho_f , \qquad \lim_{|x| \to \infty} \bar U(x) = 0 .
\ee
Thus $\bar U$ is exactly the electrostatic potential we would have in the case of the classical Vlasov-Poisson system. The remainder $\widehat U$ must satisfy either
\be \label{eq:hatU-var}
\Delta \widehat U = g e^{\bar U + \widehat U},
\ee
in the case of variable total charge, or
\be \label{eq:hatU-fixed}
\Delta \widehat U = \frac{ g e^{\bar U + \widehat U}}{\int_{\br^3} g e^{\bar U + \widehat U} \di x},
\ee
in the case of fixed total charge.

In the rest of this section, we show that $\bar U$ and $\widehat U$ exist and exhibit regularity estimates for them.

\subsection{Singular Part}

We recall some basic estimates on $\bar U$ satisfying the Poisson equation \eqref{eq:Poisson-sing}, in the case where $\rho_f \in L^1 \cap L^{5/3}(\RR^3)$. This is the degree of integrability we expect to have on $\rho_f$ when $f$ is a solution of the VPME system, based on the conservation of mass and energy.

To study $\bar U$, we make use of the Green's function for the Laplace equation on $\RR^3$, which is the function
\be
G(x) = \frac{1}{4 \pi |x|} , \quad x \neq 0.
\ee
The Poisson equation \eqref{eq:Poisson-sing} has a distributional solution of the form $G \ast \rho_f$ (see for example \cite[Theorem 6.21]{Lieb-Loss}). This solution decays at infinity and thus is the unique such solution by Liouville's theorem for harmonic functions.

We have the following integrability estimates on $\bar U$, which follow from \cite[Section 4.5]{Hormander}.

\begin{lem} \label{lem:Ubar}
Let $\rho_f \in L^1 \cap L^{\frac{5}{3}}(\br^3)$. Then $\bar U \in L^{3, \infty} \cap L^\infty (\br^3)$ with the estimates
\be
\| \bar U \|_{L^{3, \infty}(\br^3)} \leq C \|\rho_f \|_{L^1(\br^3)}, \quad \| \bar U \|_{L^\infty(\br^3)} \leq C \|\rho_f \|^{\frac{5}{6}}_{L^{\frac{5}{3}}(\br^3)} \| \rho_f \|^{\frac{1}{6}}_{L^1(\br^3)}, \quad [ \bar U ]_{C^{0,\frac{1}{5}}(\br^3)} \leq C \|\rho_f \|_{L^{\frac{5}{3}}(\br^3)}.
\ee 

Let $\rho_f \in L^1 \cap L^p(\br^3)$, where $p \in (1, 3)$. Then
\be
\| \bar E \|_{L^{\frac{3}{2}, \infty}(\br^3)} \leq C \| \rho_f \|_{L^1(\br^3)}, \qquad \| \bar E \|_{L^q(\br^3)} \leq C \| \rho_f \|_{L^p(\br^3)},
\ee
where
\be
\frac{1}{q} = \frac{1}{p} -\frac{1}{3} .
\ee

\end{lem}

Note in particular that for $p = \frac{5}{3}$, we have $q = \frac{15}{4}$. We thus expect to control $\bar E$, uniformly in time, in the spaces $L^{\frac{3}{2}, \infty}(\br^3)$ and $L^{\frac{15}{4}}(\br^3)$.

\subsection{Existence of the Smooth Part}

\subsubsection{Variable Total Charge}

We prove the existence of $\widehat U$ by making use of techniques from the calculus of variations. 
Consider the functional
\be
J_V[h] : = \int_{\br^3} |\nabla h (x)|^2 + g(x) e^{h(x) + \bar U(x)} \di x  \geq 0 .
\ee
The idea is to minimise $J_V$ over those functions $h$ decaying at infinity for which $\nabla h \in L^2(\RR^3)$. Note that, by a Sobolev inequality, these functions belong to $L^6(\br^3)$. Hence we introduce the following classical notation:
$$
\dot W^{1,2}(\RR^3):=\{h :\br^3\to \br\,:\, h \in L^6(\br^3),\,\nabla h \in L^2(\br^3)\}.
 $$

\begin{lem} \label{lem:var-Uhat-exist}
Assume that $\bar{U}\in\dot{W}^{1,2}(\RR^3).$ There exists a unique minimiser of $J_V$ over $\dot W^{1,2}(\RR^3)$.
\end{lem}
\begin{proof}
Consider a minimising sequence $(h_n)_n \subset \dot W^{1,2}(\RR^3)$. For sufficiently large $n$ we have the bound
{
\be
J_V [h_n] \leq J_V{[- \bar U]} = \int_{\br^3} | \nabla \bar U|^2 \di x + \int_{\br^3} g(x) \di x .
\ee }
It follows that $(\nabla h_n)_n$ is uniformly bounded in $L^2(\br^3)$. We may therefore pass to a subsequence such that $h_n \rightharpoonup h$ in $L^6(\br^3)$ and $\nabla h_n \rightharpoonup \nabla h$  in $L^2(\br^3)$.
Also, by the Rellich-Kondrachov theorem, for any bounded set $A$ the sequence $h_n \mathbbm{1}_A$ converges to $h\mathbbm{1}_A$ strongly in $L^p(\br^3)$ for any $p<6$.
Hence, by a diagonal argument, it follows that (by passing to a further subsequence) we may assume that $h_n$ converges to $h$ almost everywhere on $\br^3$.

By lower semi-continuity of the norm under weak convergence, we have
\be
\int_{\br^3} | \nabla h|^2 \di x \leq \liminf_{n \to \infty} \int_{\br^3} | \nabla h_n|^2 \di x .
\ee
By Fatou's lemma, we have
\be
\int_{\br^3} g e^{h + \bar U} \di x = \int_{\br^3} \lim_{n \to \infty} g e^{h_n + \bar U} \di x \leq \liminf_{n \to \infty}  \int_{\br^3} g e^{h_n + \bar U} \di x .
\ee
It follows that
\be
J_V[h] \leq \liminf_{n \to \infty} J_V[h_n] = \inf_{\phi} J_V[\phi] .
\ee
Thus $h$ is a minimiser. The uniqueness of $h$ follows from the convexity of $J_V$.
\end{proof}

We now show that the smooth part of the potential $\widehat U$ can be taken to be the minimiser of $J_V$.
Let $\widehat U$ denote the minimiser of $J_V$ and note that
\be
\int_{\RR^3} g e^{\widehat U + \bar U} \di x \leq J_V[\widehat U] \leq J_V[- \bar U]
\ee
and thus $ge^{\widehat U + \bar U}$ is a function in $L^1(\RR^3)$.

It is then possible to show that $\widehat U$ satisfies
\be
\Delta  \widehat U = g e^{\bar U + \widehat U} ,
\ee
which is the Euler-Lagrange equation associated to the minimisation problem above (see Appendix B).

\subsubsection{Fixed Total Charge} \label{sec:fixed-exist}

In this subsection we prove the existence of $\widehat U$ in the case of fixed total charge.
We will use an estimate due to Bouchut \cite[Lemma 2.6]{Bouchut}, which is used to obtain lower bounds on the integral
\be
\int_{\RR^3} g e^{U} \di x.
\ee
This will provide upper bounds on the nonlinearity in the Poisson equation in the fixed total charge case.

\begin{lem} \label{lem:Bouchut-Marc}
Let $g \in L^1 \cap L^\infty(\br^3)$ with $\int_{\br^3} g \di x = 1$. Then, for $U \in L^{3,\infty}(\br^3)$, the following estimate holds:
\be
\int_{\br^3} g e^{-|U|} \di x \geq C e^{- C \| U \|_{L^{3, \infty}(\br^3)} \, \| g \|^{\frac{1}{3}}_{L^\infty(\br^3)}} .
\ee
\end{lem}

We recall that $\bar U$ has the representation $G \ast \rho_f$ and is therefore non-negative in the cases we consider ($d=3$).

\begin{lem}[Existence of $\widehat U$]
Let $\bar U \in \dot W^{1,2}(\br^3)$ be non-negative. Then there exists a unique solution $\widehat U \in \dot W^{1,2}(\br^3)$ satisfying
\be
\Delta \widehat U = \frac{g e^{\widehat U + \bar U}}{ \int_{\br^3} g e^{\bar U + \widehat U} \di x} .
\ee
For this $\widehat U$, we have
\be
0 <  \int_{\br^3} g e^{\bar U + \widehat U} \di x < + \infty .
\ee
\end{lem}
\begin{proof}
The uniqueness of solutions in the class $\dot W^{1,2}(\br^3)$ follows from \cite[Lemma 2.5]{Bouchut}. To construct a solution, we look for a minimiser of
\be
J_F [h] : = \int_{\br^3} |\nabla h|^2 \di x + \log{ \left ( \int_{\br^3} g e^{\bar U + h} \di x \right ) } .
\ee
The difficulty in this case compared to the variable charge case is that this functional is not bounded below. We therefore introduce an approximating functional $J_K$,defined by
\be
J_K[h] : = \int_{\br^3} |\nabla h|^2 \di x + L_K\left ({ \int_{\br^3} g e^{\bar U + h} \di x }\right) .
\ee
The function $L_K$ is a smooth and non-decreasing approximation of the logarithm function, satisfying
\be
L_K(x) : = \begin{cases}
\log{x} & x > e^{-(K-1)} \\
-K & x \leq e^{-K} 
\end{cases}, \qquad |L_K'(x)| \leq \frac{1}{x} \wedge e^{K-1} \qquad \| L_K'' \|_{L^\infty} \leq C_K .
\ee

We minimise $J_K$ over the space $\dot W^{1,2}(\br^3)$. First, note that
\be
\inf J_K[h] \leq J_K[-\bar U] = \| \nabla \bar U \|_{L^2(\br^3)}^2 + L_K\(\| g \|_{L^1(\br^3)}\) .
\ee
Let $(h_n)_n$ be a minimising sequence. Since $L_K$ is bounded from below by $-K$, we have the uniform estimates
\begin{align}
\| \nabla h_n \|^2_{L^2(\br^3)} \leq \| \nabla \bar U \|_{L^2(\br^3)}^2 + L_K(\| g \|_{L^1(\br^3)}) + K \\
\int_{\br^3} g e^{\bar U + h_n}  \di x \leq e^{-(K-1)} \vee \exp{\left [ \| \nabla \bar U \|_{L^2(\br^3)}^2 + L_K(\| g \|_{L^1(\br^3)}) \right ]}.
\end{align}
As in the proof of Lemma~\ref{lem:var-Uhat-exist}, we may pass to a subsequence such that $h_n$ converges almost everywhere to some $h^{(K)}$, with $\nabla h_n$ converging weakly in $L^2(\br^3)$ to $\nabla h^{(K)}$. Therefore
\be
\| \nabla h^{(K)} \|_{L^2(\br^3)} \leq \liminf_{n \to \infty} \| \nabla h_n \|_{L^2(\br^3)}, \qquad \int_{\br^3} g e^{\bar U + h^{(K)}} \di x  \di x \leq \liminf_{n \to \infty} \int_{\br^3} g e^{\bar U + h_n}  \di x.
\ee
Since $L_K$ is smooth and increasing, we have that
\be
J_K[h^{(K)}] \leq \liminf_{n \to \infty} J_K[h_n] = \inf_h J_K[h].
\ee
Hence $h^{(K)}$ is a minimiser of $J_K$. It follows that $h^{(K)}$ is a solution of the associated Euler-Lagrange equation
\be \label{Poisson-K}
\Delta h^{(K)} = g e^{\bar U + h^{(K)}} \, L_K' \left ( \int_{\br^3} g e^{\bar U + h^{(K)}} \di x \right ) .
\ee

The right hand side of the approximating Poisson equation \eqref{Poisson-K} is non-negative and its $L^1$ norm satisfies
\be
\int_{\br^3} g e^{\bar U + h^{(K)}} \, L_K' \left ( \int_{\br^3} g e^{\bar U + h^{(K)}} \di x \right ) \di x \leq M_K\left (  \int_{\br^3} g e^{\bar U + h^{(K)}} \di x \right ),
\ee
where $M_K$ denotes the function
\be
M_K(x) = x L_K'(x) .
\ee
By assumption on $L_K$, $|M_K| \leq 1$. Therefore $\Delta h^{(K)} \in L^1(\br^3)$ with
\be
\| \Delta h^{(K)} \|_{L^1(\br^3)} \leq 1 .
\ee
It follows that there exists $C$ independent of $K$ such that
\be
\| h^{(K)} \|_{L^{3, \infty}(\br^3)} \leq C.
\ee
Therefore, by Lemma~\ref{lem:Bouchut-Marc},
\be
\int_{\br^3} g e^{\bar U + h^{(K)}} \di x \geq \int_{\br^3} g e^{h^{(K)}} \di x \geq C_{g} > 0,
\ee
where $C_{g}$ depends only on $g$, and in particular is independent of $K$. We may choose $K$ sufficiently large such that $e^{-(K-1)} < C_g$. This implies that
\be
L_K' \left ( \int_{\br^3} g e^{\bar U + h^{(K)}} \di x \right ) = \frac{1}{ \int_{\br^3} g e^{\bar U + h^{(K)}} \di x},
\ee
so that for this choice of $K$, $h^{(K)}$ is in fact a solution of \eqref{eq:Poisson-fixed}. We let $\widehat U = h^{(K)}$.

\end{proof}

\subsection{Regularity of the Smooth Part}

In this subsection, we prove regularity estimates on $\widehat U$.

\subsubsection{Variable Total Charge}

We prove the following regularity estimates on the function $\widehat U$, constructed above as the unique minimiser of $J_V$ over $\dot{W}^{1,2}(\RR^3)$.

\begin{prop}\label{prop:variable-hatU-reg}
Let $\rho \in L^1 \cap L^{\frac{5}{3}}(\br^3)$. Let $\bar U = G \ast \rho$. Then there exists $\widehat U$ satisfying \eqref{eq:hatU-var} and the estimates
\begin{align}
\| \widehat U \|_{L^{3, \infty}} &\leq C \| g \|_{L^1(\br^3)} \exp\left\{C\| \rho \|_{L^1(\br^3)}^{\frac{1}{6}}\| \rho \|_{L^\frac{5}{3}(\br^3)}^{\frac{5}{6}}\right\} \\
\| \nabla \widehat U \|_{L^{\frac{3}{2}, \infty}} &\leq C  \| g \|_{L^1(\br^3)}\exp\left\{C\| \rho \|_{L^1(\br^3)}^{\frac{1}{6}}\| \rho \|_{L^\frac{5}{3}(\br^3)}^{\frac{5}{6}}\right\} \\
\| \widehat U \|_{C^{1, \alpha}} & \leq C  \| g \|_{L^\infty(\br^3)}\exp\left\{C\| \rho \|_{L^1(\br^3)}^{\frac{1}{6}}\| \rho \|_{L^\frac{5}{3}(\br^3)}^{\frac{5}{6}}\right\}, \qquad \text{ for all } \alpha \in (0,1) .\\
\end{align}
\end{prop}
These estimates will follow from standard regularity theory for the Poisson equation, provided that we can prove suitable integrability estimates on $g e^{\bar U + \widehat U}$. To do this, we first find a representation for $\widehat U$ in terms of the Green's function $G$. First recall that $\widehat U$ satisfies the equation
\be
\Delta \widehat U = g e^{\bar U + \widehat U} .
\ee
Then note that the following convolution with $G$ is a solution of the same equation:
\be
- G \ast (g e^{\bar U + \widehat U}).
\ee
Since $g e^{\bar U + \widehat U} \in L^1$, this convolution belongs to the space $L^{3, \infty}(\RR^3)$. Thus the difference $- G \ast (g e^{\bar U + \widehat U}) - \widehat U$ is a harmonic function decaying at infinity. Then by Liouville's theorem
\be \label{eq:hatU-conv-G}
\widehat U = - G \ast (g e^{\bar U + \widehat U}) .
\ee

From this representation it follows that $\widehat U \leq 0$.
In particular,
\be
g e^{\bar U + \widehat U} \leq g e^{\bar U} .
\ee
Then, for all $p \in [1, + \infty]$,
\begin{align} \label{est:Poi-var-Lp}
\| g e^{\bar U + \widehat U} \|_{L^p(\RR^3)} & \leq \| g e^{\bar U} \|_{L^p(\RR^3)} \\
& \leq e^{\| \bar U \|_{L^\infty(\RR^3)}} \| g \|^{1 - 1/p}_{L^\infty(\RR^3)} \| g \|_{L^1(\RR^3)}^{1/p} < + \infty .
\end{align}

Using this, we may deduce the following lemma.

\begin{lem} \label{lem:hatU-reg}
Assume that $\bar U \in L^\infty(\br^3)$. Then $\widehat U \in L^{3, \infty}\cap C^{1, \alpha}(\br^3)$ for all $\alpha \in (0,1)$, with the estimates
\be
\| \widehat U \|_{L^{3, \infty}(\br^3)} \leq C e^{\| \bar U \|_{L^{\infty}(\br^3)}} \| g \|_{L^1(\br^3)}, \qquad \| \widehat U \|_{C^{1,\alpha}(\br^3)} \leq C \| g \|_{L^\infty(\br^3)} \, e^{\| \bar U \|_{L^\infty(\br^3)}} .
\ee
\end{lem}
\begin{proof}
We use the representation \eqref{eq:hatU-conv-G} in combination with the $L^p$ estimates \eqref{est:Poi-var-Lp}.

In the case $p = 1$, we have
\be
\| \Delta \widehat U \|_{L^1(\RR^3)} \leq e^{\| \bar U \|_{L^\infty(\RR^3)}} \| g \|_{L^1(\RR^3)}
\ee
By \cite[Section 4.5]{Hormander}, $\widehat U \in L^{3, \infty}(\RR^3)$ and $\widehat E \in L^{\frac{3}{2}, \infty} \cap L^\infty(\br^3)$, with
\be
\| \widehat U \|_{L^{3, \infty}(\br^3)} \leq C e^{\| \bar U \|_{L^\infty(\RR^3)}} \| g \|_{L^1(\RR^3)} , \quad
\| \widehat E \|_{L^{\frac{3}{2}, \infty}(\br^3)} \leq C e^{\| \bar U \|_{L^\infty(\RR^3)}} \| g \|_{L^1(\RR^3)} .
\ee

In the case $p = \infty$, we have 
\be \label{est:hatU-var-linfty}
\| \Delta \widehat U \|_{L^\infty(\RR^3)} \leq e^{\| \bar U \|_{L^\infty(\RR^3)}} \| g \|_{L^\infty(\RR^3)} .
\ee
By \cite[Section 4.5]{Hormander}, $\widehat E \in C^{0, \alpha}(\br^3)$ for all $\alpha \in (0,1)$, with
\be
\| \widehat E \|_{C^{0,\alpha}(\br^3)} \leq C_g e^{\| \bar U \|_{L^\infty(\br^3)}} .
\ee

\end{proof}

\subsubsection{Fixed Total Charge}

In this case, $\widehat U$ satisfies
\be \label{eq:Poisson-fixed}
\Delta \widehat U = \frac{g e^{\widehat U + \bar U}}{\int_{\br^3} g e^{\widehat U + \bar U} \di x} .
\ee
We will perform a similar analysis as in the variable charge case above. The idea is to prove integrability estimates for $\Delta \widehat U$. In the fixed charge case, we always have
\be
\| \Delta \widehat U \|_{L^1(\RR^3)} = 1 .
\ee
This implies that $\widehat U \in L^{3, \infty}(\br^3)$ and that for some universal constant $C$,
\be \label{est:Uhat-fixed-L1-unif}
\| \widehat U \|_{L^{3, \infty}(\br^3)} \leq C, \quad \| \widehat E \|_{L^{\frac{3}{2}, \infty}(\br^3)} \leq C  .
\ee

We next consider an $L^\infty$ estimate.
Once again, we have the representation of $\widehat U$ in terms of a convolution with the fundamental solution $G$. This representation implies that $\widehat U \leq 0$, and so
\be
g e^{\bar U + \widehat U} \leq g e^{\| \bar U\|_{L^\infty(\RR^3)}} .
\ee
In order to prove an $L^\infty$ estimate on $\Delta \widehat U$, the remaining step is to find a lower bound for the integral
\be
\int_{\br^3} g e^{\widehat U + \bar U} \di x .
\ee

To do this, we use the fact that $\bar U \geq 0$ to deduce that
\be
\int_{\br^3} g e^{\widehat U + \bar U} \di x \geq \int_{\br^3} g e^{\widehat U} \di x .
\ee
Then, by estimate \eqref{est:Uhat-fixed-L1-unif} and Lemma~\ref{lem:Bouchut-Marc}, there exists a constant $C_g > 0$ depending on $g$ only such that
\be
\label{eq:g}
\int_{\br^3} g e^{\widehat U + \bar U} \di x \geq C_g > 0.
\ee
Thus
\be \label{est:hatU-fixed-linfty}
\| \Delta \widehat U \|_{L^\infty(\RR^3)} \leq C_g e^{\| \bar U \|_{L^\infty(\br^3)}} .
\ee
From these estimates we deduce the following proposition.

\begin{prop} \label{prop:fixed-hatU-reg}
Let $\rho \geq 0$ satisfy $\| \rho \|_{L^1(\br^3)} = 1$ and $\rho \in L^{\frac{5}{3}}(\br^3)$. Let $\bar U$ be the unique $\dot W^{1,2}(\br^3)$ solution of \eqref{eq:Poisson-sing}. Then there exists a solution of \eqref{eq:Poisson-fixed}, which satisfies for all $\alpha \in (0,1)$,
\be
\| \widehat U \|_{L^{3,\infty}(\br^3)} \leq C, \quad \| \widehat E \|_{L^{\frac{3}{2},\infty}(\br^3)} \leq C, \quad \| \widehat U \|_{C^{1,\alpha}(\br^3)} \leq \exp{\left [C_{\alpha,g} \left (\| \rho \|^{\frac{5}{6}}_{L^{\frac{5}{3}}(\br^3)}\right ) \right ]} .
\ee
\end{prop}

This is proved using the same Sobolev embedding estimates as in the variable charge case, using the corresponding $L^p$ estimates on $\Delta \widehat U$ proved above.

\subsection{Stability estimates} \label{sec:electric-stability}

We want to extend to the VPME setting the uniqueness results in the style of Loeper for the case of solutions with $\rho_f \in L^\infty(\br^3)$. For this, we will need some stability estimates for the electrostatic potential with respect to the charge density. The aim of this section is to prove the following results.

\begin{prop}[Stability estimates: variable total charge]
\label{prop:stable variable}
Let $\rho_1, \rho_2 \in L^\infty(\br^3)$ be probability densities on $\br^3$. Let $\bar U_i \in \dot W^{1,2} \cap L^\infty(\br^3)$ solve respectively for $i=1,2$
\be
-\Delta \bar U_i = \rho_i .
\ee
Let $\widehat U_i \in L^{3, \infty} \cap L^\infty \cap \dot W^{1,2}(\br^3)$ satisfy
\be
\Delta \widehat U_i = g e^{\widehat U_i + \bar U_i} .
\ee
Then
\begin{align}
&\| \nabla \bar U_1 - \nabla \bar U_2 \|_{L^2(\br^3)} \leq \max_i \| \rho_i \|^{\frac{1}{2}}_{L^\infty(\br^3)} W_2(\rho_1, \rho_2), \\
& \| \nabla \widehat U_1 - \nabla \widehat U_2 \|_{L^2(\br^3)} \leq C \max_i \| \rho_i \|^{\frac{1}{2}}_{L^\infty(\br^3)} W_2(\rho_1, \rho_2) ,
\end{align}
where
\be
C = \| g\|^{\frac12}_{L^{\frac{3}{2}}(\br^3)} \exp{\left \{ C_0 \left [ 1 + \max_i{\| \bar U_i \|_{L^\infty(\br^3)}} + \max_i{\| \widehat U_i \|_{L^\infty(\br^3)}} \right ] \right \}} .
\ee
\end{prop}

\begin{prop}[Stability estimates: fixed total charge]
\label{prop:stable fixed}
Let $\rho_1, \rho_2 \in L^\infty(\br^3)$ be probability densities on $\br^3$. Let $\bar U_i \in \dot W^{1,2} \cap L^\infty(\br^3)$ solve respectively for $i=1,2$
\be
-\Delta \bar U_i = \rho_i .
\ee
Let $\widehat U_i \in L^{3, \infty} \cap L^\infty \cap \dot W^{1,2}(\br^3)$ satisfy
\be
\Delta \widehat U_i = \frac{g e^{\widehat U_i + \bar U_i}}{\int_{\br^3} g e^{\widehat U_i + \bar U_i} \di x } .
\ee
Then
\begin{align}
& \| \nabla \bar U_1 - \nabla \bar U_2 \|_{L^2(\br^3)} \leq \max_i \| \rho_i \|^{\frac{1}{2}}_{L^\infty(\br^3)} W_2(\rho_1, \rho_2), \\
& \| \nabla \widehat U_1 - \nabla \widehat U_2 \|_{L^2(\br^3)} \leq C \max_i \| \rho_i \|^{\frac{1}{2}}_{L^\infty(\br^3)} W_2(\rho_1, \rho_2) ,
\end{align}
where
\be
C = \| g\|^{\frac12}_{L^{\frac{3}{2}}(\br^3)} \exp{\left \{ C_0 \left [ 1 + \max_i{\| U_i \|_{L^\infty(\br^3)}} + \max_i{\| \widehat U_i \|_{L^\infty(\br^3)}} \right ] \right \}} .
\ee
\end{prop}

To prove these results, we first recall the following estimate from  \cite[Theorem 2.9]{Loeper}. Note that in the original statement it is assumed that the densities have finite second moments.
However, by approximation, this assumption can be dropped.

\begin{lem}[Stability for $\bar U$]
Let $\rho_1, \rho_2 \in L^\infty(\br^3)$ be probability densities on $\br^3$. Let $\bar U_i$ solve respectively for $i=1,2$
\be
-\Delta \bar U_i = \rho_i, \qquad \bar U_i(x) \to 0 \text{ as } |x| \to \infty .
\ee
Then
\be
\| \nabla \bar U_1 - \nabla \bar U_2 \|_{L^2(\br^3)} \leq \max_i \| \rho_i \|^{\frac{1}{2}}_{L^\infty(\br^3)} W_2(\rho_1, \rho_2) .
\ee
\end{lem}

The next step is to control the smoother part of the potential in terms of the singular part.

\subsubsection{Variable Total Charge}

{

\begin{lem}[Stability for $\widehat U$: variable total charge]
Let $\phi_1, \phi_2 \in L^{3,\infty} \cap L^\infty \cap \dot W^{1,2}(\br^3)$ be given non-negative functions. Let $\widehat U_1, \widehat U_2 \in L^{3, \infty} \cap L^\infty \cap \dot W^{1,2}(\br^3)$ satisfy
\be
\Delta \widehat U_i = g e^{\widehat U_i + \phi_i}, \; i=1,2 .
\ee
Then
\be
\| \nabla \widehat U_1 - \nabla \widehat U_2 \|^2_{L^2(\br^3)} \leq C \, \| \nabla \phi_1 - \nabla \phi_2 \|_{L^2(\br^3)}^2 ,
\ee
where, for some uniform constant $C_0$,
\be
C = \| g\|_{L^{\frac{3}{2}}(\br^3)} \exp{\left \{ C_0 \left [ 1 + \max_{i=1,2}{\| \phi_i \|_{L^\infty(\br^3)}} + \max_{i=1,2}{\| \widehat U_i \|_{L^\infty(\br^3)}} \right ] \right \}} .
\ee

\end{lem}
\begin{proof}
Consider the difference $\widehat U_1 - \widehat U_2$, which satisfies the equation
\be \label{eq:diff-var}
\Delta (\widehat U_1 - \widehat U_2) = g \left (e^{\widehat U_1 + \phi_1} - e^{\widehat U_2 + \phi_2} \right ) .
\ee
Using $\widehat U_1 - \widehat U_2$ as a test function in the weak form of \eqref{eq:diff-var}, we find that
\begin{align}
\| \nabla( \widehat U_1 - \widehat U_2) \|_{L^2(\br^3)}^2 & = \int_{\br^3} g  \left (e^{\widehat U_2 + \phi_2} - e^{\widehat U_1 + \phi_1} \right ) (\widehat U_1 - \widehat U_2) \di x \\
& =  \int_{\br^3} g e^{\widehat U_2} \left ( e^{\phi_2} - e^{\phi_1} \right ) (\widehat U_1 - \widehat U_2) \di x + \int_{\br^3} g e^{\phi_1} \left ( e^{\widehat U_2} - e^{\widehat U_1} \right ) (\widehat U_1 - \widehat U_2) \di x .
\end{align}
It is valid to use $\widehat U_1 - \widehat U_2$ as a test function since $\widehat U_1 - \widehat U_2 \in \dot W^{1,2}(\br^3)$ and $g  \left (e^{\widehat U_2 + \phi_2} - e^{\widehat U_1 + \phi_1} \right ) \in L^1 \cap L^\infty(\br^3)$.

For all $x,y \in \br$, by the Mean Value Theorem there exists $\xi \in (x,y)$ such that
\be
e^x - e^y = (x-y)e^\xi .
\ee
We therefore have the two inequalities
\be \label{ineq:exp-below}
(e^x - e^y)(x-y) \geq |x-y|^2 e^{\min{\{ x,y\}}}
\ee
and
\be \label{ineq:exp-above}
|e^x - e^y| \leq |x-y| e^{\max{\{x,y\}}} .
\ee

Since $\widehat U_i \leq 0$, we have the estimate
\be
\| \nabla( \widehat U_1 - \widehat U_2) \|_{L^2(\br^3)}^2 \leq \frac{C_{\phi_1, \phi_2}}{C_{\widehat U_1, \widehat U_2} } \int_{\br^3} ge^{\widehat U_2} |\phi_1 - \phi_2| |\widehat U_1 - \widehat U_2 | \di x - \frac{C_{\widehat U_1, \widehat U_2}}{C_{\phi_1, \phi_2}} \int_{\br^3} g e^{\phi_1} |\widehat U_1 - \widehat U_2|^2 \di x ,
\ee
where
\be
C_{\phi_1, \phi_2} = \exp{\left ( \max_{i=1,2}{\|\phi_i\|_{L^\infty(\br^3)}} \right )}, \qquad C_{\widehat U_1, \widehat U_2} = \exp{\left (- \max_{i=1,2}{\|\widehat U_i \|_{L^\infty(\br^3)}} \right )} .
\ee
Using Young's inequality for products, with a small parameter, we obtain for any $\eta > 0$ 
\be
\| \nabla( \widehat U_1 - \widehat U_2) \|_{L^2(\br^3)}^2 \leq \frac{C_{\phi_1, \phi_2}}{4 \eta} \| \phi_1 - \phi_2 \|_{L^2(g)}^2 + \left ( \eta C_{\phi_1, \phi_2} - C_{\widehat U_1, \widehat U_2} \right )  \| \widehat U_1 - \widehat U_2 \|_{L^2(g)}^2 .
\ee
Taking $\eta$ such that $\eta C_{\phi_1, \phi_2} =C_{\widehat U_1, \widehat U_2},$ we conclude that
\be
\| \nabla( \widehat U_1 - \widehat U_2) \|_{L^2(\br^3)}^2 \leq \frac{1}{4}\frac{C^3_{\phi_1, \phi_2}}{C^2_{\widehat U_1, \widehat U_2}} \, \| \phi_1 - \phi_2 \|_{L^2(g)}^2 .
\ee
We may then apply H\"{o}lder and Sobolev inequalities to obtain
\be
\| \nabla( \widehat U_1 - \widehat U_2) \|_{L^2(\br^3)}^2 \leq C \, \| \nabla \phi_1 - \nabla \phi_2 \|_{L^2(\br^3)}^2 .
\ee
where
\be
C = \| g\|_{L^{\frac{3}{2}}(\br^3)} \exp{\left \{ C_0 \left [ 1 + \max_{i=1,2}{\| \phi_i \|_{L^\infty(\br^3)}} + \max_{i=1,2}{\| \widehat U_i \|_{L^\infty(\br^3)}} \right ] \right \}} .
\ee

\end{proof}

\subsubsection{Fixed Total Charge}
\begin{lem}[Stability for $\widehat U$: fixed total charge] \label{lem:fixed-stability}
Let $\phi_1, \phi_2 \in L^{3,\infty} \cap L^\infty \cap \dot W^{1,2}(\br^3)$, $\phi_1, \phi_2 \geq 0$ be given. Let $\widehat U_1, \widehat U_2 \in L^{3, \infty} \cap L^\infty \cap \dot W^{1,2}(\br^3)$ satisfy
\be \label{eq:Poisson-hats-fixed}
\Delta \widehat U_i = \frac{g e^{\widehat U_i + \phi_i}}{\int_{\br^3} g e^{\widehat U_i + \phi_i} \di x } .
\ee
Then
\be
\| \nabla \widehat U_1 -  \nabla \widehat U_2 \|^2_{L^2(\br^3)} \leq C \, \| \nabla \phi_1 - \nabla \phi_2 \|^2_{L^2(\br^3)},
\ee
where, for some uniform constant $C_0$,
\be
C = \| g\|_{L^{\frac{3}{2}}(\br^3)} \exp{\left \{ C_0 \left [ 1 + \max_{i=1,2}{\left \{ \| \phi_i \|_{L^\infty(\br^3)}, \| \widehat U_i \|_{L^\infty(\br^3)} \right\}} \right ] \right \}} .
\ee
\end{lem}
\begin{proof}
The difference $\widehat U_1 - \widehat U_2$ satisfies
\be \label{eq:diff-fixed}
\Delta( \widehat U_1 - \widehat U_2) = \frac{g e^{\widehat U_1 + \phi_1}}{\int_{\br^3} g e^{\widehat U_1 + \phi_1} \di x } - \frac{g e^{\widehat U_2 + \phi_2}}{\int_{\br^3} g e^{\widehat U_2 + \phi_2} \di x } .
\ee
We introduce the notation
\be
m_i : = \int_{\br^3} g e^{\widehat U_i + \phi_i} \di x .
\ee
We have the estimates
\be
\max_{i=1,2} \| \widehat U_i \|_{L^{3,\infty}(\br^3)} \leq C,
\ee
since the right hand sides of the equations \eqref{eq:Poisson-hats-fixed} each have total integral equal to one.
By Lemma~\ref{lem:Bouchut-Marc}, and since $\widehat U_i \leq 0$ and $\phi_i \geq 0$, we have uniform upper and lower bounds
\be \label{m-lower}
\| g \|_{L^1(\br^3)} e^{\| \phi_i \|_{L^\infty(\br^3)}} \geq m_i \geq e^{-C} .
\ee

From the weak form of equation \eqref{eq:diff-fixed}, for all $\chi \in C^\infty_c(\br^3)$,
\be \label{eq:diff-fixed-wk}
- \int_{\br^3} \nabla \chi \cdot \nabla (U_1 - U_2) \di x = \int_{\br^3} \chi \left [ \frac{g e^{\widehat U_1 + \phi_1}}{m_1 } - \frac{g e^{\widehat U_2 + \phi_2}}{m_2 }  \right ] \di x
\ee
From the assumptions on $\widehat U_i$, we deduce that the right hand side of \eqref{eq:diff-fixed} is uniformly bounded in $L^\infty$ and $L^1$. We can therefore extend the weak form \eqref{eq:diff-fixed-wk} to test functions $\chi \in \dot W^{1,2}(\br^3)$. We may therefore choose $\chi = \widehat U_1 - \widehat U_2$, which results in the identity
\begin{align}
\| \nabla (\widehat U_1 - \widehat U_2) \|^2_{L^2(\br^3)} & =  \int_{\br^3} (\widehat U_1 - \widehat U_2)\left [ \frac{g e^{\widehat U_2 + \phi_2}}{m_2 } - \frac{g e^{\widehat U_1 + \phi_1}}{m_1 }  \right ] \di x \\
& =  \int_{\br^3} g \left [ (\widehat U_1 + \phi_1 - \log{m_1}) - ( \widehat U_2 + \phi_2 - \log{m_2}) \right ] \left [ \frac{e^{\widehat U_2 + \phi_2}}{m_2 } - \frac{e^{\widehat U_1 + \phi_1}}{m_1 }  \right ] \di x   \\
&\quad - \int_{\br^3} g \left ( \phi_1 - \phi_2 \right ) \left [ \frac{e^{\widehat U_2 + \phi_2}}{m_2 } - \frac{e^{\widehat U_1 + \phi_1}}{m_1}  \right ] \di x - \log{\left ( \frac{m_2}{m_1} \right ) } \int_{\br^3} g  \left [ \frac{e^{\widehat U_2 + \phi_2}}{m_2 } - \frac{e^{\widehat U_1 + \phi_1}}{m_1 }  \right ] \di x .
\end{align}
The final term is equal to zero, by definition of $m_i$. Applying the inequalities \eqref{ineq:exp-below} and \eqref{ineq:exp-above} above results in the inequality
\begin{align}
\| \nabla (\widehat U_1 - \widehat U_2) \|^2_{L^2(\br^3)} &\leq - c_1 \|  (\widehat U_1 + \phi_1 - \log{m_1}) - ( \widehat U_2 + \phi_2 - \log{m_2}) \|^2_{L^2(g)} \\
&\quad + C_1 \int_{\br^3} g \left | \phi_1 - \phi_2 \right | \left | (\widehat U_1 + \phi_1 - \log{m_1}) - ( \widehat U_2 + \phi_2 - \log{m_2})\right | \di x,
\end{align}
where
\be
C_1 = \frac{e^{\max_{i=1,2}{\|\widehat U_i + \phi_i \|_{L^\infty(\br^3)}}}}{\min_{i=1,2}{ m_i}}, \qquad c_1 = \frac{e^{- \max_{i=1,2}{\| \widehat U_i +\phi_i \|_{L^\infty(\br^3)}}}}{\max_{i=1,2}{m_i}} .
\ee
Young's inequality for products, with a parameter, then implies the following estimate for any $\alpha > 0$:
\be
\| \nabla (\widehat U_1 - \widehat U_2) \|^2_{L^2(\br^3)} \leq \left ( \frac{C_1}{4 \alpha} - c_1 \right ) \|  (\widehat U_1 + \phi_1 - \log{m_1}) - ( \widehat U_2 + \phi_2 - \log{m_2}) \|^2_{L^2(g)} + C_1 \alpha \| \phi_1 - \phi_2 \|^2_{L^2(g)} .
\ee
Choosing $\alpha = \frac{C_1}{4 c_1}$ gives
\be
\| \nabla (\widehat U_1 - \widehat U_2) \|^2_{L^2(\br^3)} \leq \frac{C_1^2}{4 c_1}  \| \phi_1 - \phi_2 \|^2_{L^2(g)} .
\ee
Then, since $g \in L^1 \cap L^\infty$, we deduce that
\be
\| \nabla (\widehat U_1 - \widehat U_2) \|^2_{L^2(\br^3)} \leq C_g \| \phi_1 - \phi_2 \|^2_{L^6(\RR^3)} \leq C  \| \nabla \phi_1 - \nabla \phi_2 \|^2_{L^2(\RR^3)},
\ee
where, for some universal constant $C_0 > 0$,
\be
C = \| g \|_{L^{\frac{3}{2}}(\br^3)} \exp{\left \{ C_0 \left [ 1 + \max_{i=1,2}{\left\{\| \phi_i \|_{L^\infty(\br^3)}, \| \widehat U_i \|_{L^\infty(\br^3)} \right\}} \right ] \right \}} .
\ee

\end{proof}
}

\section{Uniqueness} \label{sec:uniqueness}

In this section we prove the uniqueness and stability in $W_2$ of solutions to $(VPME)_V$ and $(VPME)_F$ with bounded density.
Recall that, given two non-negative measures on $\br^d$ with the same mass, one defines
\be
W_2^2(\mu, \nu) : = \inf_{\pi \in \Pi(\mu,\nu)} \int_{\br^d\times \br^d} |x-y|^2 \pi(\di x \di y) ,
\ee
where $\pi \in \Pi(\mu,\nu)$ denotes the set of all probability measures in $\br^{2d}$ that have marginals $\mu$ and $\nu$.

Although the strategy of proof is very similar to the one used in our paper \cite{IGP-WP}, the fact of working in the whole space requires some modifications.
The proof will be identical for the two models \eqref{eq:vpme-var} and \eqref{eq:vpme-fixed}, so we state it as a single theorem.

\begin{thm}[Uniqueness for solutions with bounded density]
\label{thm:uniq}
Let ${f_0 \in L^1(\br^3 \times \br^3)}$ be a probability density with $\rho_{f_0} \in L^\infty(\br^3)$.
Fix a final time $T > 0$, and assume that $g\in L^1\cap L^\infty(\br^3)$, with $g\geq 0$ satisfying $\int_{\br^3}g=1$. Then there exists at most one solution ${f\in C([0,T] ; L^1(\br^3 \times \br^3))}$ of  \eqref{eq:vpme-var}  (resp. \eqref{eq:vpme-fixed}) with initial datum $f_0$ such that $\rho_f \in L^{\infty}([0,T] ; L^\infty(\br^3))$.

Moreover, the following quantitative stability estimate holds. Let $f_i$, $i=1,2$, be two solutions of  \eqref{eq:vpme-var}  (resp. \eqref{eq:vpme-fixed}) with $\rho_{f_i} \in L^{\infty}([0,T] ; L^\infty(\br^3))$.Then there exists a constant $C $, depending only on $g$ and on $\sup_{t\in [0,T]} \left(\|\rho_{f_i}(t)\|_{L^1(\br^3)}+\|\rho_{f_i}(t)\|_{L^\infty(\br^3)}\right)$ ($i=1,2$),
 such that for all $t \in [0,T]$ the following bound holds:
\begin{enumerate}
\item If $W_2(f_1(0), f_2 (0)) > 1/2$ then
$$
W_2(f_1 (t), f_2 (t)) \leq W_2(f_1(0), f_2 (0))e^{Ct}.
$$
\item If $W_2(f_1(0), f_2 (0)) \leq 1/2$, let $t_0>0$ be such that
$  \log[W_2(f_1(0), f_2 (0))]e^{- C t_0}=\log(1/2)$.
Then
\begin{equation} \label{str-str-stab-thm}
W_2(f_1(t), f_2 (t)) \leq \begin{cases}
\exp{\left [ \log[W_2(f_1 (0), f_2 (0))]\, e^{- C t}\right]}& \text{for }t \in [0,t_0] \\
\frac12 e^{C(t-t_0)} & \text{for }t \in [t_0,T] .
\end{cases}
\end{equation}
\end{enumerate}

\end{thm}

\begin{proof}
Let $f_1,f_2 \in C([0,T] ; L^1(\br^3 \times \br^3))$ be two solutions of \eqref{eq:vpme-var}  (resp. \eqref{eq:vpme-fixed}) such that
$\rho_{f_1},\rho_{f_2}\in L^{\infty}([0,T] ; L^\infty(\br^3))$.

We will prove the result by means of a Gronwall type estimate. To do this, we note that as in \cite{IGP-WP}, thanks to our assumptions on the density, the electric field is log-Lipschitz
and therefore our solutions are transported by their respective characteristics, that we denote by $(X^{(1)},V^{(1)})$ and $(X^{(2)},V^{(2)})$.

Fix an arbitrary initial coupling $\pi_0 \in \Pi \left (f_1 (0), f_2 (0) \right )$ and consider the quantity
\be \label{def:D}
D(t) : = \int |X^{(1)}_t - X^{(2)}_t|^2 +  |V^{(1)}_t - V^{(2)}_t|^2 \di \pi_0 .
\ee
As in \cite{IGP-WP}, it follows from the definition of Wasserstein distance that
\be \label{D-ctrl-W}
W_2^2\(\rho_{f_1} (t), \rho_{f_2} (t)\) \leq W_2^2\(f_1 (t), f_2 (t)\) \leq D(t) .
\ee
Moreover, since $\pi_0$ was arbitrary, we have
\be \label{W0-ctrl-D0}
W_2^2\(f_1 (0), f_2 (0)\) = \inf_{\pi_0} D(0) .
\ee
Hence, it suffices to control $D(t)$. This amounts to performing a Gronwall estimate along the trajectories of the characteristic flow.

Differentiating with respect to $t$ gives
\be \label{D-derivative}
\dot D(t) = 2 \int_{(\br^3 \times \br^3)^2} (X^{(1)}_t - X^{(2)}_t) \cdot (V^{(1)}_t - V^{(2)}_t ) + (V^{(1)}_t - V^{(2)}_t ) \cdot \[ E_{1,t}(X^{(1)}_t) - E_{2,t}(X^{(2)}_t) \] \di \pi_0 
\ee
We split the electric field into four parts:
\begin{align}
E_{1,t}(X^{(1)}_t) - E_{2,t}(X^{(2)}_t) & = \[ \bar E_{1,t}(X^{(1)}_t) - \bar E_{1,t}(X^{(2)}_t) \] + \[ \bar E_{1,t}(X^{(2)}_t) - \bar E_{2,t}(X^{(2)}_t) \]  \\
& \qquad + \[ \widehat E_{1,t}(X^{(1)}_t) - \widehat E_{1,t}(X^{(2)}_t) \] + \[ \widehat E_{1,t}(X^{(2)}_t) - \widehat E_{2,t}(X^{(2)}_t)\] .
\end{align}
Applying H\"older's inequality to \eqref{D-derivative}, we obtain
\be
\dot D \leq D + 2 \sqrt{D} \sum_{i=1}^4 I_i^{1/2} ,
\ee
where
\begin{align} \label{def:Ii}
\begin{split}
&I_1(t) := \int_{(\br^3 \times \br^3)^2} |\bar E_{1,t}(X^{(1)}_t) - \bar E_{1,t}(X^{(2)}_t)|^2 \di \pi_0, \quad I_2(t) := \int_{(\br^3 \times \br^3)^2} |\bar E_{1,t}(X^{(2)}_t) - \bar E_{2,t}(X^{(2)}_t)|^2 \di \pi_0; \\
&I_3(t):= \int_{(\br^3 \times \br^3)^2} |\widehat E_{1,t}(X^{(1)}_t) - \widehat E_{1,t}(X^{(2)}_t)|^2 \di \pi_0, \quad I_4(t):= \int_{(\br^3 \times \br^3)^2} |\widehat E_{1,t}(X^{(2)}_t) - \widehat E_{2,t}(X^{(2)}_t)|^2 \di \pi_0.
\end{split}
\end{align}
We estimate the above terms in Lemmas~\ref{lem:I1}-\ref{lem:I4} below. Altogether we obtain
\be
\dot D \leq \begin{cases} C  D \lvert \log(D) \rvert & \text{ if } D < 1/2\\
C   D & \text{ if } D \geq 1/2 .
\end{cases}
\ee
Therefore
\be
D(t) \leq \exp{\left[ \log( D(0)) e^{- C  t} \right]} 
\ee
as long as $D(t) \leq 1/2$, while once $D(t)$ reaches $1/2$ (say at some time $\bar t\geq 0$) then we have the alternative bound
$$
D(t) \leq \frac12 e^{C (t-\bar t)} .
$$
From these bounds, the stability follows. 
\end{proof}

In the remainder of this section, we prove Lemmas~\ref{lem:I1}-\ref{lem:I4}.
We shall need the regularity estimates on $\bar E$ provided by the boundedness of the density.
It will be convenient to state them in a rather unusual but compact form, for later use in Lemmas \ref{lem:I1} and \ref{lem:I3}.

\begin{lem}\label{lem:elliptic}
Let $\bar U:=G\ast \rho$, where $G=-\frac{1}{4\pi|x|}$ is the Green function,
and assume that $\|\rho\|_{L^1(\br^3)}+\|\rho\|_{L^\infty(\br^3)}\leq M$ for some $M \geq 1$.
Let $H:\br^+\to \br^+$ denote the function defined as
$$
H(s) := \begin{cases}
s \left (\log s \right )^2 &\text{ if } s \leq e^{-2} \\
4e^{-2}& \text{ if } s > e^{-2}.
\end{cases}
$$
Then there exists a universal constant $C$ such that
$$
|\nabla \bar U(x)-\nabla \bar U(y)|^2\leq C\,M^2\, H(|x-y|^2)\qquad \text{for all }x,y \in \br^3.
$$
\end{lem}
\begin{proof}
Let $\bar U_M:=\frac1{M}\bar U$ and $\rho_M:=\frac1M \rho$, and note that 
$\bar U_M:=G\ast \rho_M$ with $\|\rho_M\|_{L^1(\br^3)}+\|\rho_M\|_{L^\infty(\br^3)}\leq 1$.

Hence, applying  \cite[Lemma 3.1]{Loeper} to the function $\bar U_M$ we deduce that 
\be
\| \nabla \bar U_M\|_{L^\infty(\br^3)}\leq C, \quad |\nabla \bar U_M(x)-\nabla \bar U_M(y)|\leq C|x-y|\bigl|\log|x-y|\bigr|\quad \text{for all }x,y \in \br^3\text{ with }|x-y| \leq e^{-1}.
\ee
This estimate implies that
$$
|\nabla \bar U_M(x)-\nabla \bar U_M(y)|^2\leq C\,H(|x-y|^2)\qquad \text{for all }x,y \in \br^3,
$$
and recalling that $\bar U_M=\frac1{M}\bar U$, this concludes the proof.
\end{proof}

In all the following lemmas, $D(t)$ is defined as in \eqref{def:D}.
\begin{lem}[Control of $I_1$] \label{lem:I1}
Let $I_1$ be defined as in \eqref{def:Ii}. Then
\be
I_1(t) \leq C H(D(t)),
\ee
where $H$ is defined in Lemma \ref{lem:elliptic}.
\end{lem}
\begin{proof}
Since the density associated to $\rho_{f_1}$ is uniformly bounded, we can apply Lemma \ref{lem:elliptic} to bound
\begin{align}
I_1(t) &\leq   C \int_{(\br^3 \times \br^3)^2} H\left(|X^{(1)}_t - X^{(2)}_t|^2\right) \di \pi_0 .
\end{align}
Also, one can check that the function $H$ is concave on $\mathbb R^+$.
Thus, since $\pi_0$ is a probability measure, we may apply Jensen's inequality to deduce that
\be
I_1(t)  \leq  C  \, H\left( \int_{(\br^3 \times \br^3)^2} |X^{(1)}_t - X^{(2)}_t|^2 \di \pi_0 \right) \leq C\, H(D(t)),
\ee
where the last inequality follows from the fact that $H$ is non-decreasing.
\end{proof}

\begin{lem}[Control of $I_2$] \label{lem:I2}
Let $I_2$ be defined as in \eqref{def:Ii}. Then
\be
I_2(t) \leq  C\, D(t) .
\ee
\end{lem}
\begin{proof}
One can note that, for any test function $\phi$,
\be
\int_{(\br^3 \times \br^3)^2} \phi(X^{(i)}_t) \di \pi_0 = \int_{\br^3 \times \br^3} \phi(x) f_i (t, x,v) \di x \di v \label{pushforward2}= \int_{\br^3} \phi(x) \rho_{f_i} (t, x) \di x .
\ee
Thus
\begin{align}
I_2(t) & = \int_{\br^3} |\bar E_{1,t}(x) - \bar E_{2,t}(x)|^2 \rho_{f_2} (t,x) \di x \leq \lVert \rho_{f_2}(t)  \rVert_{L^{\infty}(\br^3)} \lVert \bar E_{1,t} - \bar E_{2,t}\rVert_{L^2(\br^3)}^2,
\end{align}
and we conclude using Propositions~\ref{prop:stable variable}-\ref{prop:stable fixed} (depending on the model under consideration) and \eqref{D-ctrl-W}.
\end{proof}

{

\begin{lem}
[Control of $I_3$] \label{lem:I3}
Let $I_3$ be defined as in \eqref{def:Ii}. Then
\be
I_3(t) \leq C H(D(t)),
\ee
where the constant $C > 0$ depends only on $\mc{E}[f_1(0)]$.
\end{lem}
\begin{proof}
Note that
\be
\label{eq:hat U C2a}
\Delta \widehat U_{1,t} =g e^{\widehat U_{1,t} + \bar U_{1,t}}
\qquad \biggl(\text{resp. }\Delta \widehat U_{1,t} = \frac{g e^{\widehat U_{1,t} + \bar U_{1,t}}}{\int_{\br^3} g e^{\widehat U_{1,t} + \bar U_{1,t}} \di x }\biggr).
\ee
We can thus deduce a log-Lipschitz estimate on $\widehat U$ by using Lemma~\ref{lem:elliptic}. To do this we therefore need $L^1$ and $L^\infty$ estimates on $\Delta \widehat U$.

By \eqref{est:hatU-var-linfty} and \eqref{est:hatU-fixed-linfty} 
\be 
\| \Delta \widehat U_{1,t} \|_{L^\infty(\RR^3)} \leq C_g e^{\| \bar U_{1,t} \|_{L^\infty(\br^3)}} .
\ee
Then, using the $L^\infty$ estimate on $\bar U$ from Lemma~\ref{lem:Ubar},
\be
\| \Delta \widehat U_{1,t} \|_{L^\infty(\RR^3)} \leq C_g \exp \left [C \|\rho_{f_1} (t)\|^{\frac{5}{6}}_{L^{\frac{5}{3}}(\br^3)} \| \rho_{f_1}(t) \|^{\frac{1}{6}}_{L^1(\br^3)} \right ] \leq C,
\ee
where $C$ depends only on the initial datum $f_1(0)$. 

For the $L^1$ estimates, in the fixed charge case we always have
\be
\| \Delta \widehat U_{1,t} \|_{L^\infty(\RR^3)} = 1 .
\ee
In the variable charge case, by \eqref{est:hatU-var-linfty} we have
\be
\| \Delta \widehat U_{1,t} \|_{L^1(\RR^3)} \leq  \| g \|_{L^1(\RR^3)} e^{\| \bar U_{1,t} \|_{L^\infty(\br^3)}} \leq \| g \|_{L^1(\RR^3)} \exp \left [C \|\rho_{f_1}(t) \|^{\frac{5}{6}}_{L^{\frac{5}{3}}(\br^3)} \| \rho_{f_1}(t) \|^{\frac{1}{6}}_{L^1(\br^3)} \right ] \leq C ,
\ee
where $C$ depends only on the initial datum $f_1(0)$.

Therefore, by Lemma~\ref{lem:elliptic},
\be
|\nabla \widehat U_{1,t}(x)-\nabla \widehat U_{1,t}(y)|^2\leq C H(|x-y|^2)\qquad \text{for all }x,y \in \br^3,
\ee
for some $C$ depending only on $f_1(0)$.

We then argue as in Lemma~\ref{lem:I1}: using the above regularity estimate on $\nabla \widehat U_{1,t}$, we have
\be
I_3(t) \leq   C \int_{(\br^3 \times \br^3)^2} H\left(|X^{(1)}_t - X^{(2)}_t|^2\right) \di \pi_0 .
\ee
Since $H$ is concave and non-decreasing,
\be
I_3(t) \leq C  \, H\left( \int_{(\br^3 \times \br^3)^2} |X^{(1)}_t - X^{(2)}_t|^2 \di \pi_0 \right) \leq C\, H(D(t)) .
\ee
which concludes the proof.

\end{proof}
}

\begin{lem}[Control of $I_4$] \label{lem:I4}
Let $I_4$ be defined as in \eqref{def:Ii}. Then
\be
I_4(t)  \leq C\, D(t),
\ee
where $D$ is defined as in \eqref{def:D} and $C_{M,d}$ depends on $M$ and $d$.
\end{lem}
\begin{proof}
Using \eqref{pushforward2}, we deduce that
$$
I_4 (t) = \int_{\br^3} |\widehat E_{1,t}(x) - \widehat E_{2,t}(x)|^2 \rho _{f_2}(t,x) \di x  \leq \lVert \rho _{f_2}(t) \rVert_{L^{\infty}(\br^3)} \lVert \widehat E_{1,t} - \widehat E _{2,t} \rVert^2_{L^2(\br^3)},
$$
and we conclude by Propositions~\ref{prop:stable variable}-\ref{prop:stable fixed} and \eqref{D-ctrl-W}.
\end{proof}

\section{Moment Estimates} \label{sec:moments}

In this section, we turn to the existence of strong solutions.
We adapt the method of construction of solutions developed by Lions and Perthame \cite{Lions-Perthame} for the Vlasov-Poisson system for electrons. 
The key step is to prove an a priori estimate on the velocity moments of a solution. This is the content of the following proposition.

\begin{prop} \label{prop:moment-control}
Let $f_0 \in L^{\infty}(\br^3 \times \br^3)$, $f_0 \geq 0$, $\| f_0 \|_{L^1(\br^3 \times \br^3)} = 1$. Assume that $f_0$ also satisfies, for some $m_0 > 3$,
\be
M_{m_0}(0) : = \int_{\br^3 \times \br^3}|v|^{m_0} f_0(x,v) \di x \di v < + \infty .
\ee
Let $f$ be a solution of \eqref{eq:vpme-var} (resp. \eqref{eq:vpme-fixed})
such that for all $t$,
\be
\mc{E}_V[ f] (t) \leq C \qquad \text{(resp.  $\mc{E}_F[ f] (t) \leq C$)},
\ee
and satisfying
\be
\| f(t, \cdot ,\cdot) \|_{L^\infty(\br^3 \times \br^3)} \leq \| f_0 \|_{L^\infty(\br^3 \times \br^3)}, \qquad \int_{\br^3 \times \br^3} f(t,x,v) \di x \di v = \int_{\br^3 \times \br^3} f_0(x,v) \di x \di v = 1 .
\ee
Then, for all $k < m_0$,
\be
\int_{\br^3 \times \br^3} |v|^k f(t,x,v) \di x \di v \leq \exp{\left [C  \bigl ( 1 + \log(1+{M_k(0)}) \bigr ) \exp{(Ct)}\right ]},
\ee
for some constant $C$ depending only on $ \| f_0 \|_{L^\infty(\br^3 \times \br^3)}$ and $\mc{E}_V[ f_0]$ (resp. $\mc{E}_F[ f_0] $).
\end{prop}

The proof of this proposition follows by the classical argument by Lions and Perthame \cite{Lions-Perthame}, simplified with the use of Lorentz spaces. Indeed, since our electric field can be split as the sum of the classical Vlasov-Poisson field plus a term uniformly bounded in $L^{\frac{3}{2},\infty}(\br^3)\cap L^\infty(\br^3),$ the proof in \cite{Lions-Perthame} can be adapted with some minor modifications.
For convenience of the interested reader, we detail the argument in Appendix A.

\appendix
\section{Proof of Proposition \ref{prop:moment-control}.}
The aim is to control the velocity moments $M_k$ by use of a Gronwall estimate, where
\be
M_k(t) : = \sup_{0 \leq s \leq t} \int_{\br^3 \times \br^3} |v|^k f(s,x,v) \di x \di v .
\ee
From \cite{Lions-Perthame}, by using the equation we can deduce the estimate
\be \label{Moment-Gronwall-1}
\frac{\di }{\di t} M_k(t) \leq C \| E_t \|_{L^{k+3}(\br^3)} M_k(t)^{\frac{k+2}{k+3}} .
\ee
It therefore remains to control $\| E_t \|_{L^{k+3}(\br^3)}$. We assume from now on, without loss of generality, that $k > 3$.

First, we note that the conservation of energy gives us uniform in time bounds on $\rho_f$ and therefore $E$. By Lemma~\ref{energy-moment} and conservation of mass, we have the uniform bounds
\be
\| \rho_f (t, \cdot) \|_{L^1(\br^3)} \equiv 1 , \qquad \sup_{t} \| \rho_f (t, \cdot) \|_{\frac{5}{3}(\br^3)} \leq C .
\ee
From the regularity estimates above, we deduce that we have uniform bounds on the electric field:
\begin{align}
\sup_t \| \bar E_t \|_{L^{\frac{3}{2}, \infty}(\br^3)} \leq C, & \qquad \sup_t \| \bar E_t \|_{L^{\frac{15}{4}}(\br^3)} \leq C \\
\sup_t \| \widehat E_t \|_{L^{\frac{3}{2}, \infty}(\br^3)} \leq C, & \qquad \sup_t \| \widehat E_t \|_{L^{\infty}(\br^3)} \leq C .
\end{align}

If $k+3 > \frac{15}{4}$, we require further estimates on $\| \bar E_t \|_{L^{k+3}(\br^3)}$. To do this, we will follow the strategy of \cite{Lions-Perthame}. 
We first note some preliminary estimates relating the $L^p(\br^3)$ norms of $\rho_f$ and $\bar E$ and similar quantities to moments of $f$.

\begin{lem} \label{lem:rho-interpolation}
For any $s,t \geq 0$ and $k \geq 0$,
\be
\left \| \int_{\br^3} f(s,\cdot-vt, v) \di v  \right \|_{L^{\frac{k+3}{3}}(\br^3)} \leq M_k(s)^{\frac{3}{k+3}} .
\ee
\end{lem}
\begin{proof}
This is a standard interpolation argument. For any $R > 0$,
\be
\int_{\br^3} f(s, x-vt, v) \di v \leq R^{-k} \int_{|v| > R} |v|^k f(s, x-vt, v) \di v + \| f(s, \cdot, \cdot ) \|_{L^\infty(\br^3\times\br^3)} R^3 .
\ee
Optimising over $R$ gives
\be
\int_{\br^3} f(s, x-vt, v) \di v \leq C \left ( \int_{|v| > R} |v|^k f(s, x-vt, v) \di v \right )^{\frac{3}{k+3}} .
\ee

Then 
\be
\left \| \int_{\br^3} f(s,\cdot-vt, v) \di v  \right \|_{L^{\frac{k+3}{3}}(\br^3)} \leq \left ( \int_{\br^3} |v|^k f(s, x-vt, v)\di x \di v  \right )^{\frac{3}{k+3}} \leq M_k(s)^{\frac{3}{k+3}}.
\ee
\end{proof}

Using Lemma~\ref{lem:Ubar}, we deduce that control of moments implies integrability of $\bar E$.
\begin{lem} \label{lem:Ebar}
Let $n \in (0,6)$ and $q \in (\frac{3}{2}, + \infty)$ satisfy
\be
q = \frac{3}{6-n} \cdot (n+3) .
\ee
Then there exists a constant $C_q > 0$ such that, for all $t\ge 0,$
\be
\| \bar E_t \|_{L^q(\br^3)} \leq C_q M_n^{\frac{3}{n+3}}(t) .
\ee
\end{lem}

The resulting estimate on $\| \bar E_t \|_{L^{k+3}}$ is not sufficient to allow us to obtain a long term estimate from the differential inequality \eqref{Moment-Gronwall-1}. The next step is to obtain an improved estimate on $\bar E$.
We start by obtaining a formula for $\rho_f$ by solving the equation along characteristics with $-E\cdot \na_v f$ as a source term. From a Duhamel representation of $f$, we deduce as in \cite{Lions-Perthame} that
\be
\rho_f(t,x) = - \div_x \int_0^t s \int_{\br^3} \left [ Ef(t-s, x-vs, v) \right ] \di v \di s + \int_{\br^3} f_0(x-vt,v) \di v .
\ee
Since $\bar E = \nabla \Delta^{-1} \rho_f$, by using Sobolev inequality and Calderon-Zygmund theory we deduce that
\be
\| \bar E_t \|_{L^{k+3}(\br^3)} \leq \left \| \int_0^t s \int_{\br^3} \left [ Ef(t-s, \cdot-vs, v) \right ] \di v \di s \right \|_{L^{k+3}(\br^3)} + \left \| \int_{\br^3} f_0(\cdot-vt,v) \di v \right \|_{L^{\frac{3(k+3)}{k+6}}(\br^3)}.
\ee

To estimate the term involving $f_0$, we use Lemma~\ref{lem:rho-interpolation} to deduce that
\be
\left \| \int_{\br^3} f_0(\cdot-vt,v) \di v \right \|_{L^{\frac{3(k+3)}{k+6}}(\br^3)} \leq C M_l(0)^{\frac{3}{l+3}},
\ee
where $l$ is chosen such that
\be
\frac{l+3}{3}  = \frac{3(k+3)}{k+6} = \frac{9}{k + 6} \cdot \frac{k+3}{3} .
\ee
Since we have assumed that $k > 3$, then $l < k$ and so $M_l(0)$ is controlled by $M_k(0)$.

To estimate the term involving $Ef$, we proceed as in \cite{Lions-Perthame}, and we split the time integral into a short time and a long time part:
\begin{multline}
\left \| \int_0^t s \int_{\br^3} \left [ Ef(t-s, \cdot-vs, v) \right ] \di v \di s \right \|_{L^{k+3}(\br^3)} \leq \left \| \int_0^{t_0} s \int_{\br^3} \left [ Ef(t-s, \cdot-vs, v) \right ] \di v \di s \right \|_{L^{k+3}(\br^3)} \\
+ \left \| \int_{t_0}^t s \int_{\br^3} \left [ Ef(t-s, \cdot-vs, v) \right ] \di v \di s \right \|_{L^{k+3}(\br^3)}.
\end{multline}
We complete the estimates on these terms in the following two subsections.

\subsection{Long Time Estimate}

In this subsection, we prove that
\be \label{Est:long-time}
\left \| \int_{t_0}^t s \int_{\br^3} \left [ Ef(t-s, \cdot-vs, v) \right ] \di v \di s \right \|_{L^{k+3}(\br^3)}  \leq C |\log t_0| \, M_k(t)^{\frac{1}{k+3}} .
\ee

We will require the following H\"{o}lder-like inequalities for the Lorentz spaces $L^{p,q}$ - see O'Neil \cite{ONeil} and Tartar \cite{Tartar}.
\begin{lem} \label{lem:Lorentz-ineq}
\begin{enumerate}[(i)]
\item Let $0 < p_1, p_2, p < \infty$, $0 < q_1, q_2, q \leq \infty$ satisfy $p^{-1} = p^{-1}_1 + p^{-1}_2$, $q^{-1} = q^{-1}_1 + q^{-1}_2$.
Then
\be
\| fg\|_{L^{p,q}} \leq C_{p_1, p_2, q_1, q_2} \, \| f\|_{L^{p_1,q_1}}  \| g\|_{L^{p_2,q_2}} ,
\ee
whenever the right hand side is finite.

\item Let $f \in L^1 \cap L^\infty$. Then, for any $p \in [1,\infty)$,
\be
\| f\|_{L^{p,1}} \leq C_p \| f \|^{\frac{1}{p}}_{L^1} \| f \|_{L^\infty}^{1 - \frac{1}{p}} .
\ee
\end{enumerate}

\noindent Consequently, we have the following estimate.
Let $p \in (1, \infty]$. Let $f \in L^{p, \infty}$ and $g \in L^1 \cap L^\infty$. Then
\be
\| fg \|_{L^1} \leq C_{p} \, \| f \|_{L^{p, \infty}} \| g \|_{L^1}^{1 - \frac{1}{p}} \| g \|_{L^\infty}^{\frac{1}{p}} .
\ee
\end{lem}

Using this, we prove \eqref{Est:long-time}.
By Minkowski's inequality,
\be \label{time-out}
\left \| \int_{t_0}^t s \int_{\br^3} \left [ Ef(t-s, \cdot-vs, v) \right ] \di v \di s \right \|_{L^{k+3}(\br^3)} \leq \int_{t_0}^t s \left \| \int_{\br^3} \left [ Ef(t-s, \cdot-vs, v) \right ] \di v \right \|_{L^{k+3}(\br^3)} \di s.
\ee
By Lemma~\ref{lem:Lorentz-ineq}, we have the following estimate:
\be \label{Holder-Lorentz}
\int_{\br^3} \left [ Ef(t-s, x-vs, v) \right ] \di v \leq \| E(t-s, x-s \cdot) \|_{L^{\frac{3}{2}, \infty}(\br^3)} \, \|f\|_{L^\infty(\br^3)}^{\frac{2}{3}} \left | \int_{\br^3} f(t-s, x-vs, v) \di v \right |^{\frac{1}{3}} .
\ee
By Lemma \ref{lem:Ubar}, and Propositions \ref{prop:variable-hatU-reg}, \ref{prop:fixed-hatU-reg}, $E$ is bounded in $L^{\frac{3}{2}, \infty}$, uniformly in time.
Thus
\be
\left \| \int_{t_0}^t s \int_{\br^3} \left [ Ef(t-s, \cdot-vs, v) \right ] \di v \di s \right \|_{L^{k+3}(\br^3)} \leq C \int_{t_0}^t s^{-1} \left ( \int_{\br^3 \times \br^3}  \left |  f(t-s, x-vs, v) \right |^{\frac{k+3}{3}} \di x \di v \right )^{\frac{1}{k+3}} \di s ,
\ee
where $C>0$ depends only on the initial datum.

By Lemma~\ref{lem:rho-interpolation},
\be
\left ( \int_{\br^3 \times \br^3}  \left |  f(t-s, x-vs, v) \right |^{\frac{k+3}{3}} \di x \di v \right )^{\frac{1}{k+3}} \leq M_k(t-s)^{\frac{1}{k+3}}\leq M_k(t)^{\frac{1}{k+3}},
\ee
since $s>0.$ Therefore \eqref{Est:long-time} follows.

\subsection{Short Time Estimate}

In this subsection we show that
\be
\left \| \int_0^{t_0} s \int_{\br^3} \left [ Ef(t-s, x-vs, v) \right ] \di v \di s \right \|_{L^{k+3}} \leq C t_0^{2-\frac{3}{r}} \left [ M_m(0)^{\frac{1}{k+3}} + \left ( 1 + M_k(t)^{\delta}\right ) \right ] ,
\ee
where
\be
\delta = \frac{3(m+3)}{(k+3)^2} .
\ee

By Minkowski's inequality,
\be \label{time-out-2}
\left \| \int_{0}^{t_0} s \int_{\br^3} \left [ Ef(t-s, \cdot-vs, v) \right ] \di v \di s \right \|_{L^{k+3}(\br^3)} \leq \int_0^{t_0} s \left \| \int_{\br^3} \left [ Ef(t-s, \cdot-vs, v) \right ] \di v \right \|_{L^{k+3}(\br^3)} \di s.
\ee
By H{\"{o}}lder's inequality, for any $r > \frac{3}{2}$ we obtain
\begin{align}
\int_{\br^3} \left [ Ef(t-s, x-vs, v) \right ] \di v & \leq \left (\int_{\br^3} |E(t-s, x-vs)|^r  \di v \right )^{1/r} \, \|f\|_{L^\infty(\br^3)}^{\frac{1}{r}} \left | \int_{\br^3} f(t-s, x-vs, v) \di v \right |^{1 - \frac{1}{r}} \\
& \leq s^{-\frac{3}{r}} \| E_{t-s}\|_{L^r(\br^3)} \, \|f\|_{L^\infty(\br^3)}^{\frac{1}{r}} \left | \int_{\br^3} f(t-s, x-vs, v) \di v \right |^{1 - \frac{1}{r}} . \label{Holder}
\end{align}
Thanks to Lemma \ref{lem:Ubar}, and Propositions \ref{prop:variable-hatU-reg}, \ref{prop:fixed-hatU-reg} we have
\be
\bar E_t \in L^{\frac{3}{2}, \infty} \cap L^{\frac{15}{4}}(\br^3), \qquad \widehat E_t \in L^{\frac{3}{2}, \infty} \cap L^\infty(\br^3),
\ee
with uniform in time estimates depending only on $M_2(0)$. We therefore choose $r \in ( \frac{3}{2}, \frac{15}{4})$ and obtain
\be
\left \| \int_0^{t_0} s \int_{\br^3} \left [  Ef(t-s, \cdot-vs, v) \right ] \di v \di s \right \|_{L^{k+3}(\br^3)}  \leq C \int_0^{t_0} s^{1 - \frac{3}{r}} \left \| \int_{\br^3} f(t-s, \cdot - vs, v) \di v \right \|^{\frac{1}{r'}}_{L^{\frac{k+3}{r'}}(\br^3)} \di s ,
\ee
where $r'$ satisfies $1/r + 1/r' = 1$, and the constant $C>0$ depends only on the initial datum.

To control the density term, we use Lemma~\ref{lem:rho-interpolation} with a moment of higher order than $k$. Choose $m \in (k, m_0)$ such that
\be
\frac{m+3}{3} = \frac{k+3}{r'} .
\ee
Then
\be
\int_0^{t_0} s^{1 - \frac{3}{r}} \left \| \int_{\br^3} f(t-s, \cdot - vs, v) \di v \right \|^{\frac{1}{r'}}_{L^{\frac{k+3}{r'}}(\br^3)} \di s  \leq C \int_0^{t_0} s^{1 - \frac{3}{r}} \di s \, M_m(t_0)^{\frac{1}{k+3}} \leq C t_0^{2-\frac{3}{r}} \, M_m(t_0)^{\frac{1}{k+3}}.
\ee

We control $M_m$ by using \eqref{Moment-Gronwall-1}, which implies that for all $t \geq 0$,
\be
M_m(t) \leq C \left ( M_m(0) + \left ( t \sup_{s \leq t} \| E_s \|_{L^{m+3}(\br^3)} \right )^{m+3} \right ).
\ee
$\| \widehat E_s\|_{L^{m+3}(\br^3)}$ is uniformly bounded by Lemma~\ref{lem:hatU-reg}.
For $\bar E_s$, we use Lemma~\ref{lem:Ebar} to obtain
\be
\| \bar E_s \|_{L^{m+3}(\br^3)} \leq M_n^{\frac{3}{n+3}}(s) ,
\ee
where $n=n_m \in (0,6)$ is related to $m$ via the formula
\be
m+3 = \frac{3}{6-n} \cdot (n+3) .
\ee

We now aim to control $M_n$ by $M_k$.
Note that if $n > 3$ then $n < m$. Recall that $m$ depends on $r$ and $k$, and that $m \searrow k$ as $r \searrow 3/2$.
As $m \searrow k > 3$ by assumption, $n_m\searrow n_k < k$. Therefore, by choosing $r$ sufficiently close to $3/2$, we can ensure that $n_m \leq k < m$. Then, since $M_n\leq M_k^{\frac{n}{k}}$ (by H\"older inequality), for $s \leq t$ we have
\be
\| \bar E_s \|_{L^{m+3}(\br^3)}^{m+3} \leq M_n(t)^{\frac{3(m+3)}{n+3}}\leq M_k(t)^{\frac{3n(m+3)}{k(n+3)}}
\leq (1+M_k(t))^{\frac{3n(m+3)}{k(n+3)}}\leq (1+M_k(t))^{\frac{3(m+3)}{k+3}}.
\ee
Thus
\be
M_m(t_0) \leq C \left [ M_m(0) + \left [ t_0 \left ( 1 + M_k(t)^{\frac{3}{k+3}}\right ) \right ]^{m+3} \right ].
\ee

Then, for $t_0\leq1$,
\be
\left \| \int_0^{t_0} s \int_{\br^3} \left [ Ef(t-s, x-vs, v) \right ] \di v \di s \right \|_{L^{k+3}} \leq C t_0^{2-\frac{3}{r}} \left [ M_m(0)^{\frac{1}{k+3}} + \left ( 1 + M_k(t)^{\delta}\right ) \right ] ,
\ee
where
\be
\delta = \frac{3(m+3)}{(k+3)^2} .
\ee

\subsection{Full Estimate}

Closing the estimate is identical to \cite{Lions-Perthame}. Choosing $t_0 = (1+M_k(t))^{-\frac{\delta r}{2r-3}}$, and combining all the previous estimates, gives a bound of the form
\be
 \| E_t \|_{L^{k+3}(\br^3)} \leq C ( 1 + \log{(1+M_k(t))}) (1+M_k(t))^{\frac{1}{k+3}}.
\ee
Thus, recalling \eqref{Moment-Gronwall-1}, one obtains
\be
\frac{\di}{\di t} M_k(t) \leq  C \bigl( 1 + |\log{1+M_k(t)}|\bigr) \, (1+M_k(t)) ,
\ee
which completes the proof of Proposition~\ref{prop:moment-control}.


\begin{thebibliography}{10}

\bibitem{BGNS18}
C.~Bardos, F.~Golse, T.~T. Nguyen, and R.~Sentis.
\newblock The {M}axwell-{B}oltzmann approximation for ion kinetic modeling.
\newblock {\em Phys. D}, 376/377:94--107, 2018.

\bibitem{Batt-Rein}
J.~Batt and G.~Rein.
\newblock {Global classical solutions of the periodic Vlasov-Poisson system in
  three dimensions}.
\newblock {\em C. R. Acad. Sci. Paris S{\'{e}}r. I Math.}, 313(6):411--416,
  1991.

\bibitem{BPLT1991}
G.~Bonhomme, T.~Pierre, G.~Leclert, and J.~Trulsen.
\newblock Ion phase space vortices in ion beam-plasma systems and their
  relation with the ion acoustic instability: numerical and experimental
  results.
\newblock {\em Plasma Physics and Controlled Fusion}, 33(5):507--520, may 1991.

\bibitem{Bouchut}
F.~Bouchut.
\newblock {Global weak solution of the Vlasov-Poisson system for small
  electrons mass}.
\newblock {\em Comm. Partial Differential Equations}, 16(8-9):1337--1365, 1991.

\bibitem{IGP-WP}
M.~Griffin-Pickering and M.~Iacobelli.
\newblock Global well-posedness in 3-dimensions for the {Vlasov--Poisson}
  system with massless electrons.
\newblock arXiv:1810.06928.

\bibitem{WP-proceedings}
M.~Griffin-Pickering and M.~Iacobelli.
\newblock Recent developments on the well-posedness theory for {Vlasov}-type
  equations.
\newblock arXiv:2004.01094.

\bibitem{IHK1}
D.~Han-Kwan and M.~Iacobelli.
\newblock {The quasineutral limit of the Vlasov-Poisson equation in Wasserstein
  metric}.
\newblock {\em Commun. Math. Sci.}, 15(2):481--509, 2 2017.

\bibitem{Hormander}
L.~H{\"{o}}rmander.
\newblock {\em The Analysis of Linear Partial Differential Operators I}.
\newblock Springer-Verlag, second edition, 1990.

\bibitem{Lieb-Loss}
E.~H. Lieb and M.~Loss.
\newblock {\em Analysis: Second Edition}, volume~14 of {\em Graduate Studies in
  Mathematics}.
\newblock American Mathematical Society, 2001.

\bibitem{Lions-Perthame}
P.~L. Lions and B.~Perthame.
\newblock {Propagation of moments and regularity for the 3-dimensional
  Vlasov-Poisson system}.
\newblock {\em Invent. Math.}, 105(2):415--430, 1991.

\bibitem{Loeper}
G.~Loeper.
\newblock {Uniqueness of the solution to the Vlasov-Poisson system with bounded
  density}.
\newblock {\em J. Math. Pures Appl. (9)}, 86(1):68--79, 2006.

\bibitem{Mason71}
R.~J. Mason.
\newblock Computer simulation of ion-acoustic shocks. {T}he diaphragm problem.
\newblock {\em The Physics of Fluids}, 14(9):1943--1958, 1971.

\bibitem{Medvedev2011}
Y.~V. Medvedev.
\newblock Ion front in an expanding collisionless plasma.
\newblock {\em Plasma Physics and Controlled Fusion}, 53(12):125007, nov 2011.

\bibitem{ONeil}
R.~O'Neil.
\newblock Convolution operators and {$L(p,q)$} spaces.
\newblock {\em Duke Math. J.}, 30(1):129--142, 03 1963.

\bibitem{Pfaffelmoser}
K.~Pfaffelmoser.
\newblock {Global classical solutions of the Vlasov-Poisson system in three
  dimensions for general initial data}.
\newblock {\em J. Differential Equations}, 95(2):281--303, 1992.

\bibitem{SCM}
P.~Sakanaka, C.~Chu, and T.~Marshall.
\newblock Formation of ion-acoustic collisionless shocks.
\newblock {\em The Physics of Fluids}, 14(611), 1971.

\bibitem{Tartar}
L.~Tartar.
\newblock {\em An Introduction to Sobolev Spaces and Interpolation Spaces},
  volume~3 of {\em Lecture Notes of the Unione Matematica Italiana}.
\newblock Springer-Verlag Berlin Heidelberg, 2007.

\bibitem{Ukai-Okabe}
S.~Ukai and T.~Okabe.
\newblock On classical solutions in the large in time of two-dimensional
  {Vlasov's} equation.
\newblock {\em Osaka J. Math.}, 15(2):245--261, 1978.

\end{thebibliography}
\end{document}